\documentclass[english]{article}
\usepackage[T1]{fontenc}
\usepackage[latin9]{inputenc}
\usepackage{geometry}
\usepackage{color}
\usepackage{array}
\usepackage{float}
\usepackage{amsmath}
\usepackage{amsthm}
\usepackage{amssymb}
\usepackage{graphicx}
\usepackage{setspace}
\usepackage{wasysym}
\PassOptionsToPackage{normalem}{ulem}
\usepackage{ulem}
\usepackage[noblocks]{authblk}

\makeatletter

\providecommand{\tabularnewline}{\\}
\floatstyle{ruled}
\newfloat{algorithm}{tbp}{loa}
\providecommand{\algorithmname}{Algorithm}
\floatname{algorithm}{\protect\algorithmname}

\theoremstyle{remark}
\newtheorem*{notation*}{\protect\notationname}
\theoremstyle{plain}
\newtheorem{thm}{\protect\theoremname}
\theoremstyle{plain}
\newtheorem{prop}[thm]{\protect\propositionname}
\usepackage{algolyx}
\theoremstyle{plain}
\newtheorem{lem}[thm]{\protect\lemmaname}
\theoremstyle{plain}
\newtheorem{cor}[thm]{\protect\corollaryname}
\theoremstyle{remark}
\newtheorem*{rem*}{\protect\remarkname}
\theoremstyle{definition}
\newtheorem{defn}[thm]{\protect\definitionname}
\theoremstyle{plain}
\newtheorem{assumption}[thm]{\protect\assumptionname}
\theoremstyle{remark}
\newtheorem*{note*}{\protect\notename}
\theoremstyle{plain}

\usepackage{pdfcomment}
\usepackage{environ}
\usepackage{xcolor}
\usepackage{tikz}

\@ifundefined{showcaptionsetup}{}{%
 \PassOptionsToPackage{caption=false}{subfig}}
\usepackage{subfig}
\makeatother

\usepackage{babel}
\providecommand{\algorithmname}{Algorithm}
\providecommand{\assumptionname}{Assumption}
\providecommand{\corollaryname}{Corollary}
\providecommand{\definitionname}{Definition}
\providecommand{\factname}{Fact}
\providecommand{\lemmaname}{Lemma}
\providecommand{\notationname}{Notation}
\providecommand{\notename}{Note}
\providecommand{\propositionname}{Proposition}
\providecommand{\remarkname}{Remark}
\providecommand{\theoremname}{Theorem}

\begin{document}
\title{The Bayesian Inversion Problem for Thermal Average Sampling of Quantum
Systems}

\author{Ziheng Chen\footnote{stokes615@utexas.edu}}
\affil{Department of Mathematics, The University of Texas at Austin, Austin 78705, USA}
\author{Zhennan Zhou\footnote{zhennan@bicmr.pku.edu.cn}}
\affil{Beijing International Center of Mathematical Research, Peking University, Beijing 100871, P.R. China}

\maketitle
\begin{abstract}
\begin{onehalfspace}
In this article, we propose a novel method for sampling potential
functions based on noisy observation data of a finite number of observables
in quantum canonical ensembles, which leads to the accurate sampling
of a wide class of test observables. The method is based on the Bayesian
inversion framework, which provides a platform for analyzing the posterior
distribution and naturally leads to an efficient numerical sampling algorithm.
We highlight that, the stability estimate is obtained
by treating the potential functions as intermediate variables in the
following way: the discrepancy between two sets of observation data of training
observables can bound the distance between corresponding posterior
distributions of potential functions, while the latter 
naturally leads to a bound of the discrepancies between corresponding thermal
averages of test observables. Besides, the training observables can
be more flexible than finite samples of the local density function,
which are mostly used in previous
researches. The method also applies to the multi-level quantum systems
in the non-adiabatic regime. In addition, we provide extensive numerical
tests to verify the accuracy and efficiency of the proposed algorithm.
\end{onehalfspace}
\end{abstract}

\section{Introduction}
It has long been a challenging problem to establish inversion theories
in quantum canonical ensembles (see \cite{Mehats2010}, \cite{Lemm2000a},
\cite{Lemm2000b}, \cite{J.C.LemmaJ.Uhlig}, \cite{Habeck2014}, \cite{Lemm2005}, \cite{Nguyen2017}). One of the difficulties lies in the fact that
the quantum canonical ensemble is not directly tractable as the classical
case. In this article, the following thermal average of a given observable
$\widehat{A}$ is of our interest:
\[
\langle\widehat{A}\rangle =\frac{1}{\mathcal{Z}}\mathbf{Tr}\left[\exp\left(-\beta\widehat{H}\right)\widehat{A}\right]
\]
where $\widehat{H}=-\frac{1}{2M}\Delta+{V}(q)$ is the quantum
Hamiltonian, $q \in \mathbb R^d$, $M$ is the particle mass, $V:\mathbb{R}^d\to\mathbb{R}$ is the potential function,
$\beta$ is referred to as
the inverse temperature and $\mathcal{Z}\triangleq\mathbf{Tr}\left[\exp\left(-\beta\widehat{H}\right)\right]$
is the partition function. In general, the physical observable $\widehat{A}$  corresponds to a self-adjoint operator, whereas for simplicity, we only consider observables as functions of the position variable in this work.

There have been some previous works regarding the existence of such
a potential function that yields the local density function $n$ exactly
as given. In \cite{Mehats2010}, the authors show that, in a 1-dimensional
quantum system with periodic boundary conditions, any positive density
function corresponds to a unique density operator minimizing the quantum
free energy. Furthermore, the density operator is in the form of $\exp\left(-\left(-\Delta+V\left(x\right)\right)\right)$
where $V$ is the chemical potential. The result is extended in \cite{Mehats2011}
to a multi-dimensional case as well as unbounded domains and non-linear
interactions in Hartree or Hartree-Fock systems. On the classical
counterpart, \cite{Chayes1984} addresses the problem whether there
exists an external potential corresponding to a given equilibrium
single particle density. Results are established for both the canonical
and grand canonical distributions. The Hamiltonian concerned is in
the form of $H\left(x_{1},x_{2},\dots x_{n}\right)=W\left(x_{1},x_{2},\dots x_{n}\right)+\sum_{i=1}^{n}U\left(x_{i}\right)$
where $W$ is known (not required to be symmetric) and $U$ is unknown.
Under a few integrable conditions, $U$ can be uniquely determined
by local density function. A following work \cite{Navrotskaya2014}
extends the conclusion to functions $U$ that may symmetrically contain
more than one particle coordinate $x_{i}$.

Some other works lay emphasis on recovering the potential landscape
numerically based on observations to the system. A first work \cite{Lemm2000a}
models the likelihood of potential functions  based on position
observables with finite given eigenvalues 
and the correspondingly training data $D$, which leads to an expression for the likelihood based on
the wave function. Besides, the inversion process is done by maximizing
a posteriori which involves taking the variational derivative on the
posterior probability and the numerical implementation is based on
the gradient descent algorithm. In \cite{Lemm2000b}, the effect of
consecutive measurement at non-equilibrium is taken into account.
Under the assumption of a time-independent Hamiltonian, the algorithm
proposed in the previous work is adapted and amended by considering
how the derivative of wave function depends on the time intervals between
successive measurements. In \cite{J.C.LemmaJ.Uhlig}, the authors
discussed how to design priors to recover special properties in potential
functions, such as periodicity (as for a distorted crystal surface)
or expected discontinuities. A later work \cite{Lemm2005} rewrites
the formula for the thermal average by using the Feynman path integral.
To effectively (in numerical sense) evaluate the functional derivative
of wave function, the stationary phase approximation is deployed.

The inverse problems in quantum statistics, similarly to those in other 
fields, can often be ill-posed or highly under-determined.
At the same time, there is a wide gap between the works mentioned
above: on one side, exact potential functions can be found in certain idealized scenarios, 
but only on a theoretical level; on the other side, some computational methods
have been designed, but there is no clear conclusion on numerical convergence or stability.
To merge the gaps, we adopt the Bayesian inversion framework to the
quantum thermal average problem.
In this work, a direct application on the inference of the potential functions is achieved by using this framework.
% while the special structures of the quantum thermal average as well as the Gibbs distribution on the configuration space is captured to derive useful algorithms and draw conclusions on stability.
However, due to the special structure of the quantum thermal average, i.e. the presence of
the ring polymer potential, the Bayesian sampling process is much more challenging than its
counterpart for classical systems. Thanks to the recent understanding of the continuum limit
of the ring polymer representation, the Gibbs distribution on the infinite-dimensional
configuration space serves as a crucial component in deriving efficient algorithms and
proving stability results in spite of other technical difficulties. 
Compared to deterministic approaches for
inverse problems, the Bayesian inversion method is more favorable  for infinite-dimensional problems (see e.g. \cite{Stuart2010}, \cite{Dashti2017}). The advantage of
this method is that it does not aim to identify the maximum point
of the likelihood, but rather transforms the optimization problem
into a sampling problem which naturally provides a statistical result
for the inverse problem. There are a few following works applying
this framework, for example, to the determination of the initial condition
for Navier--Stokes equation in \cite{Cotter2010}, to a three-dimensional
global seismic wave propagation inverse problem with hundreds of thousands
of parameters in \cite{Bui-Thanh2013} and to benchmarks in Higher-Order
Ice Sheet Models in \cite{Petra2014}.

The advantages of using this framework are mainly four-fold.
\begin{itemize}
\item Firstly, the Bayesian framework views the solution to the inverse
problem as a posterior distribution that modifies the a priori knowledge
of the unknown by assimilating the noisy observation data. It
is impossible to uniquely pinpoint the potential function based on only
a finite number of observations, whereas the Bayesian inversion
framework admits multiple possible solutions with varying weights
described by a posterior distribution. Based on Bayes's formula,
the inversion from the training observations to the posterior distribution
has a statistical sense. Furthermore, the stability of the
result obtained from the inversion process is also ensured on a theoretical
level.
\item Secondly, the use of such framework also naturally leads to an efficient
numerical algorithm. By building a Markov chain with the posterior
distribution as the desired invariant distribution, the potential
function can be sampled, providing further predictions on test observations.
In the numerical section, we will show that the ground truth of test
observations can be recovered if the noise covariance is assumed to
be relatively small. This result confirms that not only
the posterior distribution is stable with respect to noises, but also the predictions
on test observations are practical and meaningful. Besides, the truncation
error introduced in the ring polymer representation 
can be quantitatively estimated, which makes the ring polymer representation
 a superior asymptotic approximation compared with  the static phase approximation
adopted in \cite{Lemm2005}. 
\item Moreover, the restrictions on training observables in this work can
be largely relaxed, while in the previous works
(see \cite{Lemm2000a} and \cite{Lemm2000b}) the only viable option
is the position observable confined in a small neighborhood of the
most likely positions.
In the following theorem statement, we will
demonstrate that the training and testing observables can be expanded
to bounded continuous functions. 
This greatly helps us to relax the strict requirements applied on practical physical
measurements and ensure a wider application in real-life experiments. 
\item Last but not least, the Bayesian Inversion Framework used in this
work also helps build a seamlessly connection between theory and algorithm
in the infinite-parameter regime. There are also some other works
focusing on problems that share a similar background with this inverse
quantum thermal average problem. For example, the inverse Ising problem
has been a popular topic. In \cite{Habeck2014}, the authors adopt the 
Sequential Monte Carlo algorithm to infer the parameter $\lambda$
in systems with Hamiltonian $E\left(x\right)=\sum_{k=1}^{K}\lambda_{k}f_{k}\left(x\right)$.
In \cite{Nguyen2017}, the authors reviewed ways of recovering parameters
in the Ising-type Hamiltonian $H\left(\boldsymbol{s}\right)=-\sum_{i}h_{i}s_{i}-\sum_{i<j}J_{ij}s_{i}s_{j}$.
Methods mentioned in these works, however, only apply to a finite-parameter
regime. As is emphasized in \cite{Stuart2010}, avoiding
discretization until the last possible moment and valuing the infinite-dimensional
nature of the framework enables us to examine the coherence between
theoretical results under the continuum limit and numerical implementation
based on the ring polymer representation.
\end{itemize}
The outline of this paper is listed as follows. In Sec. \ref{sec:A-Review-on-Forward-Problem},
we review previous works addressing the sampling problem in quantum
canonical ensembles. The inversion process theory is established in
Sec. \ref{sec:Problem-Formulation}, the numerical algorithm is proposed
in Sec. \ref{sec:The-Inversion-Algorithm} and a prior analysis is
performed in Sec. \ref{sec:Prior-Analysis}. After that, we show a
few numerical studies in Sec. \ref{sec:1-level Proof-of-concept}-\ref{sec:Stability-result A-numerical-proof}
and \ref{sec:Numerical-Study-on-2-Level-System} for systems of 1
and 2 levels correspondingly. The numerical tests also help us to verify
the theoretical results proposed in Sec. \ref{sec:Prior-Analysis}
and gain further insights.

\section{Review on Quantum Thermal Average \label{sec:A-Review-on-Forward-Problem}}

First, we summarize the mathematical model and the numerical approach of the forward problem, i.e. the
thermal average of an observable in a quantum system. In the following, we introduce the ring polymer representation of the thermal average, its continuum limit and the path integral molecular dynamics method for numerical simulation.

\subsection{Ring polymer representation}

The ring polymer representation, first proposed in \cite{Feynman1965}(Sec.
10) and widely used in chemical physics (e.g. see Sec. 2.9 in \cite{Kleinert}),
approximates the thermal average of observable $\widehat{A}$ with respect to
the classical Gibbs distribution in the $dN$-dimensional space for
ring polymer $\boldsymbol{q}=\left(q_1,\dots,q_N\right)\in\mathbb{R}^{dN}, q_i\in\mathbb{R}^{d}$ as
\begin{equation}
\langle\widehat{A}\rangle =\frac{1}{\mathcal{Z}_{N}}\int_{\mathbb{R}^{dN}}\left[\frac{1}{N}\sum_{i=1}^{N}A\left(q_{i}\right)\right]e^{-S_{N}\left(\boldsymbol{q}\right)}d\boldsymbol{q}+\mathcal{O}\left(N^{-2}\right),\label{eq:1-level-thermal-average}
\end{equation}
where the action,
depending on the potential function $V:\mathbb{R}^d\to \mathbb{R}$,
is given by 
\begin{equation}
S_{N}\left(\boldsymbol{q}\right)\triangleq\beta_{N}\sum_{i=1}^{N}\left[\frac{M\left|q_{i}-q_{i+1}\right|^{2}}{2\beta_{N}^{2}}+V\left(q_{i}\right)\right]\label{eq:1-level-thermal-average-action}
\end{equation}
and $\mathcal{Z}_{N}\triangleq\int_{\mathbb{R}^{N}}e^{-S_{N}\left(q\right)}dq$
is the normalization constant and $\beta_{N}\triangleq\beta/N$.

Define the shorthand notation
$
\label{eqn:shorthand-finite}
\overline{A}\left(\boldsymbol{q}\right) \triangleq \frac{1}{N}\sum_{i=1}^{N}A\left(q_{i}\right)
$. 
Eqn. \ref{eq:1-level-thermal-average} indicates that the thermal
average $\langle\widehat{A}\rangle $ can be approximated
by the expectation of $\overline{A}\left(\boldsymbol{q}\right)$
with respect to the following Gibbs distribution
\begin{equation}
\pi_{N}^{V}\left(d\boldsymbol{q}\right)\triangleq\frac{1}{\mathcal{Z}_{N}}e^{-S_{N}\left(\boldsymbol{q}\right)}d\boldsymbol{q}\label{eq:1-level-Gibbs-distribution}.
\end{equation}

However, since the dimension of the configuration space is $dN$,
a direct numerical integration based on $\pi_{N}^{V}$ is too expensive
due to large $N$ required by reducing the approximation error. To
avoid the curse of dimensionality, a few numerical sampling methods
are designed to approximate the integration by averaging a time series
based on Eqn. \ref{eq:1-level-Gibbs-distribution}; we will further discuss this topic in Sec. \ref{Under-damped Langevin sampling}.

\subsection{Continuum limit\label{subsec:1-level Continuum-limit}}

As the division number $N$ approaches infinity, there is a formal limit
for the ring polymer configuration $\boldsymbol{q}$ as well for the
action $S_N$. For given bead number $N$, a piece-wise linear path $\mathfrak{q}_{N}$
can be constructed by setting $\mathfrak{q}_{N}\left(j\beta_{N}\right)=q_{j}$
and periodic on $\left[0,\beta\right]$, so the formal limit of $\mathfrak{q}_{N}$
as $N\to\infty$ is also in the space $\mathcal{L}\mathbb{R}^{d}\triangleq\left\{ \mathfrak{q}:\left[0,\beta\right]\to\mathbb{R}^{d},\mathfrak{q}\left(0\right)=\mathfrak{q}\left(\beta\right)\right\} $.
The corresponding limit for the action is given by
\[
S\left(\mathfrak{q}\right)\triangleq\int_{0}^{\beta}\left[\frac{M}{2}\left|\dot{\mathfrak{q}}\right|^{2}+V\left(\mathfrak{q}\left(\tau\right)\right)\right]d\tau,
\]
so the thermal average, taking the limit $N\to\infty$, can be formally
written as 
\begin{equation}
\langle\widehat{A}\rangle =\frac{1}{\mathcal{Z}}\int_{\mathcal{L}\mathbb{R}^{d}}\left[\frac{1}{\beta}\int_{0}^{\beta}A\left(\mathfrak{q}\left(\tau\right)\right)d\tau\right]\exp\left(-S\left(\mathfrak{q}\right)\right)D\left[\mathfrak{q}\right]\label{eq:1-level-thermal-average-continuum}
\end{equation}
where $D\left[\mathfrak{q}\right]$ denotes integration over all paths
$\mathfrak{q}$ with $\mathfrak{q}\left(0\right)=\mathfrak{q}\left(\beta\right)$
and $\mathcal{Z}\triangleq\int_{\mathcal{L}\mathbb{R}^{d}}e^{-S\left(q\right)}D\left[\mathfrak{q}\right]$
is the normalization constant.

Eqn. \ref{eq:1-level-thermal-average-continuum} allows us to define
the following formal Gibbs distribution (also see \cite{Feynman1965} Sec.10 and \cite{Lu2018}) on the configuration space
\begin{equation}
\pi^{V}\left(d\mathfrak{q}\right)\triangleq\frac{1}{\mathcal{Z}}\exp\left(-S\left(\mathfrak{q}\right)\right)D\left[\mathfrak{q}\right]\label{eq:1-level-Gibbs-distribution-continuum}
\end{equation}
so that the thermal average in Eq. \ref{eq:1-level-thermal-average-continuum}
can be rewritten as 
\begin{equation}
\langle\widehat{A}\rangle =\mathbb{E}_{\pi^{V}\left(d\mathfrak{q}\right)}\left[\frac{1}{\beta}\int_{0}^{\beta}A\left(\mathfrak{q}\left(\tau\right)\right)d\tau\right]=\mathbb{E}_{\pi^{V}\left(d\mathfrak{q}\right)}\overline{A}\left[\mathfrak{q}\right]\label{eq:1-level-thermal-average-continuum-expectation}
\end{equation}
where the shorthand notation is defined as
\begin{equation}
\label{eqn:shorthand-continuum}
	\overline{A}\left[\mathfrak{q}\right]\triangleq \frac{1}{\beta}\int_{0}^{\beta}A\left(\mathfrak{q}\left(\tau\right)\right)d\tau.
\end{equation}
\begin{note*}
	To emphasize the dependency of $\langle\widehat{A}\rangle $ on the potential $V$, we use the notation $G^{A}\left(V\right)$ to represent the mapping from $V$ to the thermal average $\langle\widehat{A}\rangle $ where needed. 
	On the corresponding part, the notation $G^A_N\left(V\right)$ is used to represent the mapping to the thermal average under the ring polymer representation, where $N$ is the number of beads on the ring $\boldsymbol{q}$.
\end{note*}

\subsection{Under-damped Langevin sampling\label{Under-damped Langevin sampling}}

To enhance the efficiency of the sampling process,  an auxiliary momentum variable
$\boldsymbol{p}\in\mathbb{R}^{dN}$ with artificial mass $M$ can
be introduced \cite{Liu2016} \cite{Zhang2017}. In the augmented state space of position and momentum
of ring polymer beads, the thermal average is given by

\[
\langle\widehat{A}\rangle =\frac{1}{\mathcal{Z}'_{N}}\int_{\mathbb{R}^{N}}\int_{\mathbb{R}^{N}}\left[\frac{1}{N}\sum_{i=1}^{N}A\left(q_{i}\right)\right]e^{-\beta_{N}H_{N}\left(\boldsymbol{q},\boldsymbol{p}\right)}d\boldsymbol{q}d\boldsymbol{p}+\mathcal{O}\left(N^{-2}\right)
\]
where the Hamiltonian is given by
\begin{equation}
H_{N}\left(\boldsymbol{q},\boldsymbol{p}\right)=\frac{1}{2M}\left|\boldsymbol{p}\right|^{2}+\sum_{i=1}^{N}\left[\frac{M\left|q_{i}-q_{i+1}\right|^{2}}{2\beta_{N}^{2}}+V\left(q_{i}\right)\right].\label{eq:1-level-underdamped-langevin-hamiltonian}
\end{equation}
Therefore the classical Gibbs distribution in the extended phase space
can be sampled by evolving the following dynamic system:
\begin{equation}
\begin{cases}
d\boldsymbol{q} & =\nabla_{\boldsymbol{p}}H_{N}dt\\
d\boldsymbol{p} & =-\nabla_{\boldsymbol{q}}H_{N}dt-\gamma\boldsymbol{p}dt+\sqrt{\frac{2\gamma M}{\beta_{N}}}d\boldsymbol{B}.
\end{cases}\label{eq:1-level-underdamped-langevin-equation}
\end{equation}
To numerically integrate the SODE in Eqn. \ref{eq:1-level-underdamped-langevin-equation},
we use the BAOAB method which is proposed in \cite{Leimkuhler2013},
further tested and compared against other variants (ABOBA and OBABO)
in \cite{Leimkuhler2013a} and \cite{Liu2016}, and investigated elaborately and shown useful for other types of thermostats in \cite{Zhang2017}.

\subsection{Multi-level system}

The thermal average of a given observable in a multi-level system
has been explored in the previous works. In \cite{Menzeleev2014},
kinetically-constrained RPMD is proposed to directly simulate electronically
non-adiabatic chemical processes. In \cite{Ananth2013}, the authors
proposed mapping-variable RPMD which constructs continuous Cartesian
variables for both electronic states and nuclear degrees of freedom;
see also the review paper \cite{Stock2005}.

Recently, a few works focus on the exact computation
of multi-level systems. In \mbox{\cite{Liu2018}}, three splitting
schemes for the Boltzmann operator in the non-adiabatic representation
are discussed in detail and multi-electronic-state PIMD is derived
afterwards. In \mbox{\cite{Tao2018}}, a corresponding isomorphic
Hamiltonian is introduced to fully recover the exact quantum Boltzmann
distribution under the Boltzmann sampling with classical nuclear degrees
of freedom. In this paper, we implement the method proposed by \mbox{\cite{Lu}}
where the non-adiabatic effect is added into consideration by modeling
the surface hopping procedure as a Q-process. A following work \mbox{\cite{Lu2018a}}
improves this method by introducing a multiscale integrator for the
infinite swapping limit. We will supply the essential formula in Sec.
\mbox{\ref{sec:RPR-in-Two-level-Systems}} in the Appendix.

\section{Inversion Process Based on Thermal Averages \label{sec:Problem-Formulation}}

Sec.~\ref{sec:A-Review-on-Forward-Problem} has established the forward
problem, i.e. given observable $\widehat{A}$ and potential $V$,
the thermal average can be formulated as
\[
\boldsymbol{y}_{truth}=G^{A}\left(V\right).
\]

However, an observation $\boldsymbol{y}$ from the real world or experiment may
deviate from the ground truth $G^{A}\left(V\right)$ because of noise,
possibly due to instrumental bias, measurement error, and thermal
fluctuation. A simplest model for the noise $\boldsymbol{\eta}$ is the additive
assumption, i.e.
\begin{equation}
\label{eqn:introducing-noise}
\boldsymbol{y}=G^{A}\left(V\right)+\boldsymbol{\eta},
\end{equation}
and the noise $\boldsymbol{\eta}$ is assumed independent of the potential function $V$. Since $\boldsymbol{\eta}$ is a random variable, a determinate result for the ``best'' $V$ barely makes sense. Nevertheless, it is sensible to presume a posterior distribution $\mu^{\boldsymbol{y}}\left(dV\right)$ of $V$ based on the noisy observation $\boldsymbol{y}$, given a known prior distribution on the space of all admissible potential functions.
As the observable $ \widehat{A} $ is used to recover the posterior distribution, we will name it as the ``training observable''.
On the other hand, we in practice are interested in the inference of other thermal averages within the same ensemble, which can be viewed as a weak version of the inverse problem in quantum statistics. Such concern can be generalized to the thermal average of another observable $ \widehat{O} $, named as the ``testing observable''.

Although the prediction from $ \boldsymbol{y} $ to $ \langle \widehat{O} \rangle $ may look like a classical statistics learning problem, the difficulty lies in the nonlinear nature of density operators with respect to the potential function and the observation data in quantum canonical ensembles.
To address this both theoretically and numerically, we replace the inverse mapping with the posterior distribution $\mu^{\boldsymbol{y}}\left(dV\right)$. Instead of directly computing $ G^{O}\left(V_p\right) $ by the point estimator $V_p$ (for example MAP), we treat $V$ as an intermediate random variable and thus shift our attention to the weighted average $\mathbb{E}_{\mu^{\boldsymbol{y}}\left(dV\right)}G^{O}\left(V\right)$, where the weights are obtained by sampling the posterior distribution $\mu^{\boldsymbol{y}}\left(dV\right)$ induced by the noisy training observation.
In fact, the weighted-average $\mathbb{E}_{\mu^{\boldsymbol{y}}\left(dV\right)}G^{O}\left(V\right)$ sampled by the algorithm can be written into $\mathbb{E}_{\pi^{\boldsymbol{y}}\left(d\mathfrak{q}\right)}\overline{O}\left[\mathfrak{q}\right]$, thus it suffices to study the property of the posterior distribution $\pi^{\boldsymbol{y}}\left(d\mathfrak{q}\right)$ of the conditional variable $\mathfrak{q}|\boldsymbol{y}$.

The methodology in this work contains a direct application on the Bayesian inversion framework \cite{Dashti2017} as well as some special considerations with regard to the structure of quantum thermal averages. The inversion procedure from observations $\boldsymbol{y}$ to potential functions $V$ is just a reformulation of how the posterior distribution is derived in such framework. Nevertheless, the posterior distribution on $V$ helps us to establish the posterior distribution on configurations $\mathfrak{q}$ and the weight average $\mathbb{E}_{\pi^{\boldsymbol{y}}\left(d\mathfrak{q}\right)}\overline{O}\left[\mathfrak{q}\right]$ of test observables $O$. It also helps in the following two aspects. The conclusion on stability is established by combining the theorem on stability in this framework and the fact that the Gibbs distribution continuously depends on potential $V$. The consistency of the numerical algorithm results from the equality of two thermal averages that are obtained from different perspectives.

In the following sections, we discuss the following properties
of the posterior distribution:
\begin{description}
\item [{Existence}] We formulate the inverse problem in this section.
The posterior distribution $\mu^{\boldsymbol{y}}\left(dV\right)$ can be obtained from the prior distribution $\mu_{0}\left(dV\right)$ by evaluating the negative log potential $\Phi\left(V;\boldsymbol{y}\right)$ with the training observation $\boldsymbol{y}$.
To further describe the conditional variable $\mathfrak{q}|\boldsymbol{y}$, we can view $V$ as an intermediate variable to obtain the posterior distribution $\pi^{\boldsymbol{y}}\left(d\mathfrak{q}\right)$ based on the two conditional variables $V|\boldsymbol{y}$ and $\mathfrak{q}|V$. By the end $\mathbb{E}_{\mu^{\boldsymbol{y}}\left(dV\right)}G^{O}\left(V\right)=\mathbb{E}_{\pi^{\boldsymbol{y}}\left(d\mathfrak{q}\right)}\overline{O}\left[\mathfrak{q}\right]$ is proved to confirm the consistency between our model and algorithm.
\item [{Solvability}] We demonstrate how to numerically sample the
posterior distribution $\mu^{\boldsymbol{y}} \left(dV\right) $ and discuss some details in implementation
in Sec. \ref{sec:The-Inversion-Algorithm}. The algorithm is an iterative
procedure in a proposal-decision approach: in each iteration, a new
potential proposal $\widehat{V}^{\left(k+1\right)}$ is drawn based
on the previous sample $V^{\left(k\right)}$ and the acceptance probability
is based on the comparison between current and previous observation
errors. To be more specific, the algorithm can be abstracted as:
\begin{description}
\item [{Initialization}] Obtain the ground truth observation $\boldsymbol{y}^{*}$. Draw $\widehat{V}^{\left(0\right)}=V^{\left(0\right)}\sim\mu_{0}\left(dV\right)$.
\item [{Proposal}] Draw $\widehat{V}^{\left(k+1\right)}$ based on $V^{\left(k\right)}$.
Compute training observation $\widehat{\boldsymbol{y}}^{\left(k+1\right)}=G^{A}\left(\widehat{V}^{\left(k+1\right)}\right)$.
\item [{Decision}] With probability $a\left(\widehat{\boldsymbol{y}}^{\left(k+1\right)},\boldsymbol{y}^{\left(k\right)};\boldsymbol{y}^{*}\right)$ which is consistent with the proposal scheme,
accept, otherwise reject the proposal.
\begin{description}
\item [{Accept}] $V^{\left(k+1\right)}=\widehat{V}^{\left(k+1\right)}$
and $\boldsymbol{y}^{\left(k+1\right)}=\widehat{\boldsymbol{y}}^{\left(k+1\right)}$.
\item [{Reject}] $V^{\left(k+1\right)}=V^{\left(k\right)}$ and $\boldsymbol{y}^{\left(k+1\right)}=\boldsymbol{y}^{\left(k\right)}$.
\end{description}
\end{description}
\item [{Stability}] We give some stability analysis in Sec. \ref{sec:Prior-Analysis}. We construct a series of proofs to show that the posterior distributions are stable under the disturbance of the noisy observation data.
\end{description}

The formulation of the inverse problem and the stability result will be discussed under the continuum limit, while the numerical scheme are based on the ring polymer representation in finite dimensions. Also, we assume the technical condition that the posterior distribution mentioned in the paper is absolutely continuous to the corresponding prior distribution (e.g. $\mathfrak{q}|V$ to $\mathfrak{q}$, $V|\boldsymbol{y}$ to $V$).

\subsection{Notations}

Without loss of generality, we assume that the physical space of the quantum particle is one-dimensional, i.e. $d=1$.

The space of $\beta$-periodic loops $\mathfrak{q}$ (in the continuum
limit sense) is defined as
\[
X\triangleq\mathcal{L}\mathbb{R}\triangleq\left\{ \mathfrak{q}:\left[0,\beta\right]\to\mathbb{R},\mathfrak{q}\left(0\right)=\mathfrak{q}\left(\beta\right)\right\} ,
\]
while the space of $N$-bead ring polymer $\boldsymbol{q}$ is defined as
\[
	X_N\triangleq\mathbb{R}^N.
\]

\smallskip

The training observations $\boldsymbol{y}$ based on $N_{T}$ observables
are in the space
\[
Y\triangleq\mathbb{R}^{N_{T}}.
\]

\smallskip

The potential function $V$ is identified by its difference to the harmonic potential $V_o \triangleq \frac12 x^2$, thus we define the space containing all admissible potential functions as
\[
W^{\boldsymbol{1}}\triangleq\left\{ V:\left(V-\frac{1}{2}x^{2}\right)\in L^{\infty}\left(\mathbb{R}\right)\cap L^{2}\left(\mathbb{R}\right)\right\} .
\]

Since the Hermite basis function $\lbrace \phi_i\rbrace$ (where $\phi_i\in L^2\left(\mathbb{R}\right)$, see Section \ref{sec:Analysis-Supplyments}
in Appendix for definition) is a complete orthogonal basis for $L^{2}\left(\mathbb{R}\right)$,
any given element $V\in W^{\boldsymbol{1}}$ can be written as the linear combination of the basis as
\begin{equation}
V=V_{o}+\sum_{i=0}^{\infty}v_{i}\phi_{i}\label{eq:expand-V-onto-hermite-functions}
\end{equation}
where $v_{i}=\left\langle V-V_{o},\phi_{i}\right\rangle_{L^2} $ is the
coefficient of the $i$-th basis. By Eqn. \ref{eq:expand-V-onto-hermite-functions},
we have a mapping from the space of weight sequences $\left\{ v_{i}\right\} $
to $W^{\boldsymbol{1}}$.

To make $W^{\boldsymbol{1}}$ a Banach space, we endow it with norm
\[
\left|\left|V\right|\right|_{W^{\boldsymbol{1}}}\triangleq\left|\left|V-V_{o}\right|\right|_{L^2}+\left|\left|V-V_{o}\right|\right|_{L^{\infty}}.
\]
Then $W^{\boldsymbol{1}}$ is complete (see Lem. \ref{lem:W1 complete} for proof).

In the numerical implementation, we use the truncated version of $W^{\boldsymbol{1}}$,
defined as 
\[
W_{L}^{\boldsymbol{1}}\triangleq\left\{ V:\left(V-V_{o}\right)\in\boldsymbol{span}\left(\phi_{0},\dots\phi_{L}\right)\right\} .
\]
Therefore, there exists a projection $\Pi_{L}^{\boldsymbol{1}}$ of
truncation from $W^{\boldsymbol{1}}$ to $W_{L}^{\boldsymbol{1}}$,
namely
\[
\Pi_{L}^{\boldsymbol{1}}\left(V\right)=V_{o}+\sum_{i=0}^{L}\left\langle V-V_{o},\phi_{i}\right\rangle \phi_{i}.
\]

\bigskip

Now $X,Y,W^{\boldsymbol{1}}$ are Banach spaces endowed with proper
norms. To further discuss the relation between variables on these
spaces, we need to define a few measures as the prior and posterior
distributions.

According to Eqn. \ref{eq:1-level-Gibbs-distribution-continuum}, a given potential
function $V:\mathbb{R}\to \mathbb{R}$ (recall the dimension $d$ of the physical space is assumed to be 1) induces a formal Gibbs distribution
$\pi^{V}\left(d\mathfrak{q}\right)$ on the position configuration
space $X=\mathcal{L}\mathbb{R}$.
Especially, we denote $\pi^{V_{o}}$
as $\pi_{0}$. The connection between $\pi_{0}$ and $\pi^{V}$ is
can be derived from the action $S\left(\mathfrak{q}\right)=\int_{0}^{\beta}\left[\frac{M}{2}\left|\dot{\mathfrak{q}}\right|^{2}+V\left(\mathfrak{q}\left(\tau\right)\right)\right]d\tau$,
i.e.
\begin{equation}
\pi^{V}\left(d\mathfrak{q}\right)=\frac{1}{\mathcal{Z}^V}\pi_{0}\left(d\mathfrak{q}\right)\exp\left(-\int_{0}^{\beta}\mathring{V}\circ\mathfrak{q}\,d\tau\right)\label{eq:definition =00005Cpi^v}
\end{equation}
where we use the shorthand notation $\mathring{V}\triangleq V-V_{o}=V-\frac{1}{2}x^{2}$ and $\mathcal{Z}^V\triangleq\int_X \pi_{0}\left(d\mathfrak{q}\right)\exp\left(-\int_{0}^{\beta}\mathring{V}\circ\mathfrak{q}\,d\tau\right)$ is the normalization constant.

We regard the measure $\pi_{0}$ induced by the harmonic oscillation
potential $V_{o}$ as the prior on $X$, while $\pi^{V}$ is the posterior
distribution which is absolutely continuous to $\pi_{0}$. The Radon-Nikodym (R-N for short)
derivative $\frac{1}{\mathcal{Z}^V}\exp\left(-\int_{0}^{\beta}\mathring{V}\circ\mathfrak{q}\,d\tau\right)$
can be viewed as a `correction' to the prior, adding more weight where
the correction potential $\mathring{V}$ is low.

\smallskip

To analyze the stability in the numerical algorithm, we can define the Gibbs distribution in the ring polymer representation in a similar way. By Eqn. 
\ref{eq:1-level-Gibbs-distribution}, a Gibbs distribution $\pi^V_N\left(\boldsymbol{q}\right)$ on $X_N=\mathbb{R}^N$ is defined via the action $S_N\left(\boldsymbol{q}\right)$ given the potential function $V$. Especially, we denote $\pi^{V_o}_N$ as $\pi_{0,N}$. The R-N derivative of $\pi^V_N$ with respect to $\pi_{0,N}$ is
\[
\dfrac{\text{d}\pi^V_N}{\text{d}\pi_{0,N}}\left(\boldsymbol{q}\right)=\frac{1}{\mathcal{Z}_N^V} \exp \left(-\beta_N \sum_{i=1}^N \mathring{V}\left(q_i\right)\right) \label{eq:definition rpr rn-d pi^v}
\]
where $\mathcal{Z}_N^V$ is the normalization constant.

\medskip

To give a prior distribution on the potential function space $W^{\boldsymbol{1}}$,
we construct each potential function $V$ from its component $v_{i}$
in the way that
\[
v_{i}=\gamma_{i}\xi_{i},\ \xi_{i}\overset{i.i.d.}{\sim}\mathcal{N}\left(0,1\right)
\]
where $\left\{ \gamma_{i}\right\} $ is a fixed decaying sequence.
This form naturally leads to a Gaussian measure $\mu_{0}\left(dV\right)$
on $W^{\boldsymbol{1}}$, which is defined by the mean value $\mathring{V_{o}}=0$
and the covariance operator 
\[
\Gamma_{V}\triangleq\sum_{n=0}^{\infty}\gamma_i^2 \phi_i\otimes\phi_i.
\]

The correspondence in ring polymer representation is the truncated covariance operator
\[
\Gamma_{V, L}\triangleq\sum_{n=0}^{L}\gamma_i^2 \phi_i\otimes\phi_i. \label{eq:definition truncated covariance}
\]

\medskip

Now we turn to the space of training observations $Y$. Assuming that
the noise $\boldsymbol{\eta}$ is a Gaussian noise with mean $\boldsymbol{0}$
and positive-definite covariance matrix $\Gamma_{\boldsymbol{\eta}}\in\mathbb{R}^{N_{T}\times N_{T}}$,
the distribution of $\boldsymbol{\eta}$ is clearly $\mathcal{N}\left(\boldsymbol{0},\Gamma_{\boldsymbol{\eta}}\right)$.
By the additive noise assumption, the training observation $\boldsymbol{y}$
is a translation of the ground truth $G^{A}\left(V\right)$, so the
distribution of $\boldsymbol{y}$ is $\mathcal{N}\left(G^{A}\left(V\right),\Gamma_{\boldsymbol{\eta}}\right)$
if the potential function $V$ is fixed.
To sum up, we denote $\tau_{0}\triangleq\mathcal{N}\left(\boldsymbol{0},\Gamma_{\boldsymbol{\eta}}\right)$
as the noise distribution and $\tau^{V}\triangleq\mathcal{N}\left(G^{A}\left(V\right),\Gamma_{\boldsymbol{\eta}}\right)$
as the shifted distribution dependent on $Y$.

\subsection{Formulation: the inversion process}
Recall that the goal is to sample the potential function $ V $ as well as the testing observable. As we have stated in the beginning of this section, it is equivalent to obtain the weighted average either by sampling $G^O\left(V\right)$ from $\mu^{\boldsymbol{y}}\left(dV\right)$ (i.e. the algorithm's perspective) or by sampling $\overline{O}\left[\mathfrak{q}\right]$ from $\pi^{\boldsymbol{y}}\left(d\mathfrak{q}\right)$ (i.e. viewing $V$ as an intermediate variable). To describe the posterior distribution $\pi^{\boldsymbol{y}}\left(d\mathfrak{q}\right)$, it is essential to start from the two conditional variables $ V|\boldsymbol{y} $ and $ \mathfrak{q}|\boldsymbol{y} $.
In the following reasoning, we establish the posterior distribution $ \mu^{\boldsymbol{y}}\left(dV\right) $ of $ V|\boldsymbol{y} $ in Eqn. \ref{eq:1-level definition of =00005Cmu^y} and show that the posterior measure for $\mathfrak{q}|\boldsymbol{y}$
is
\[
p(\mathfrak{q}|\boldsymbol{y})d\mathfrak{q}=\int_{W^{\boldsymbol{1}}}\left[\mu^{\boldsymbol{y}}\left(dV\right)\thinspace\pi^{V}\left(d\mathfrak{q}\right)\right].
\]

Recall that $V:\mathbb{R}\to\mathbb{R}$ is the potential function (which is in $W^{\boldsymbol{1}}$), $\boldsymbol{y}\in \mathbb{R}^{N_T}$ is the training observation and $\mathfrak{q}\in \mathcal{L}\mathbb{R}$ is the configuration. Thus, $\mathfrak{q}|V$ corresponds to the configuration variable conditioned on the given potential function of the system while $\boldsymbol{y}|V$ corresponds to the conditioned thermal average of training observables.

As we have assumed, the training observation $\boldsymbol{y}$ is
a noisy version of the exact value $G^{A}\left(V\right)$. Recall that, the noise
$\boldsymbol{\eta}$ is independent of $V$ and the distribution of $\boldsymbol{\eta}$ is $\tau_{0}=\mathcal{N}\left(\boldsymbol{0},\Gamma_{\boldsymbol{\eta}}\right)$,
so according to Eqn. \ref{eqn:introducing-noise}, the conditional distribution of $\boldsymbol{y}|V$ is
\begin{equation}
\label{eqn:distribution-for-y}
\tau^{V}=\mathcal{N}\left(G^{A}\left(V\right),\Gamma_{\boldsymbol{\eta}}\right), 
\end{equation}
merely a translation of $\tau_{0}$.

Consider the joint probability density for $\left(\mathfrak{q},\boldsymbol{y},V\right)$:
\begin{align*}
p\left(\mathfrak{q},\boldsymbol{y},V\right) & =p\left(\mathfrak{q},\boldsymbol{y}|V\right)p\left(V\right)\\
 & =p\left(\mathfrak{q}|V\right)p\left(\boldsymbol{y}|V\right)p\left(V\right)
\end{align*}
the second equation holds since 
the noise $\eta$ is assumed independent from the configuration $\mathfrak{q}$,
thus the two conditional variable $\mathfrak{q}|V$
and $\boldsymbol{y}|V$ are independent. So by using the Bayes's
formula at the following starred equation, we have
\begin{align}
p\left(\mathfrak{q},V|\boldsymbol{y}\right) & =\frac{p\left(\mathfrak{q},\boldsymbol{y},V\right)}{p\left(\boldsymbol{y}\right)}\nonumber \\
 & =p\left(\mathfrak{q}|V\right)\frac{p\left(\boldsymbol{y}|V\right)p\left(V\right)}{p\left(\boldsymbol{y}\right)}\nonumber \\
 & \overset{*}{=}p\left(\mathfrak{q}|V\right)p\left(V|\boldsymbol{y}\right).\label{eq:formulation p(q,V|y)=00003Dp(q|V)p(V|y)}
\end{align}
By taking integrals on both sides of Eqn. \ref{eq:formulation p(q,V|y)=00003Dp(q|V)p(V|y)}
on $W^{\boldsymbol{1}}$, we have
\begin{align}
p\left(\mathfrak{q}|\boldsymbol{y}\right) & =\int_{W^{\boldsymbol{1}}}p\left(\mathfrak{q},V|\boldsymbol{y}\right)dV\label{eq:formulation p(q|y)=00003Dint p(q|V) p(V|y)}\\
 & =\int_{W^{\boldsymbol{1}}}p\left(\mathfrak{q}|V\right)p\left(V|\boldsymbol{y}\right)dV.\label{eq:formulation p(q|y)=00004Dint p(q|V) p(V|y)}
\end{align}

\iffalse
We have known that the distribution of $\mathfrak{q}|V$ is $\pi^{V}$ as
in Eqn. \ref{eq:definition =00005Cpi^v}. 
\fi
By Eqn. \ref{eq:definition =00005Cpi^v} the potential function $V$, as a random variable in the measurable space $W^{\boldsymbol{1}}$, induces a probability measure $\pi^V$ on the configuration space $X$ for the conditional random variable $\mathfrak{q}|V$.
For $V|\boldsymbol{y}$,
since the R-N derivative for $\tau^{V}$ with respect to $\tau_{0}$ is 
\[
\dfrac{d\tau^{V}}{d\tau_{0}}\left(\boldsymbol{y}\right)=\frac{\exp\left(-\frac{1}{2}\left|\left|\Gamma_{\boldsymbol{\eta}}^{-\frac{1}{2}}\left(\boldsymbol{y}-G^{A}\left(V\right)\right)\right|\right|_2^{2}\right)}{\exp\left(-\frac{1}{2}\left|\left|\Gamma_{\boldsymbol{\eta}}^{-\frac{1}{2}}\boldsymbol{y}\right|\right|_2^{2}\right)},
\]
thus the negative log likelihood can be defined as
\[
\Phi\left(V;\boldsymbol{y}\right)\triangleq\frac{1}{2}\left|\left|\Gamma_{\boldsymbol{\eta}}^{-\frac{1}{2}}\left(\boldsymbol{y}-G^{A}\left(V\right)\right)\right|\right|_2^{2}-\frac{1}{2}\left|\left|\Gamma_{\boldsymbol{\eta}}^{-\frac{1}{2}}\boldsymbol{y}\right|\right|_2^{2}=-\log\left(\dfrac{d\tau^{V}}{d\tau_{0}}\left(\boldsymbol{y}\right)\right).
\]

To utilize Thm. \ref{thm:bayesian-inversion-conditional-variable}, we develop the following argument.
Denote $\nu_0$ to be the product measure $\nu_0\left(\text{d}V,\text{d}\boldsymbol{y}\right)\triangleq\mu_0\left(\text{d}V\right)\tau_{0}\left(\text{d}\boldsymbol{y}\right)$ and assume the negative log likelihood $\Phi$ is $\nu_0$-measurable.
Since the conditional variable $\boldsymbol{y}|V$ is distributed according to $\tau^V\left(\text{d}\boldsymbol{y}\right)$, the random variable pair $\left(V, y\right)$ is distributed according to $\nu\left(\text{d}V,\text{d}\boldsymbol{y}\right)\triangleq\mu_0\left(\text{d}V\right)\tau^V\left(\text{d}\boldsymbol{y}\right)$.
Furthermore, we have $\nu\ll\nu_0$ with
\[
\dfrac{\text{d}\nu}{\text{d}\nu_0}\left(V,\boldsymbol{y}\right)=\exp\left(-\Phi\left(V;\boldsymbol{y}\right)\right).
\]
Since $\nu_0$ is a product measure, the conditional distribution for $V|\boldsymbol{y}$ under $\nu_0$ naturally exists and equals $\mu_0$.
Thus, by Thm. \ref{thm:bayesian-inversion-conditional-variable}, the conditional
distribution of $V|\boldsymbol{y}$ under $\nu$ is

\begin{equation}
\mu^{\boldsymbol{y}}\left(dV\right)\triangleq\mu_{0}\left(dV\right)\frac{1}{Z\left(\boldsymbol{y}\right)}\exp\left(-\Phi\left(V;\boldsymbol{y}\right)\right)\label{eq:1-level definition of =00005Cmu^y}
\end{equation}
where $Z\left(\boldsymbol{y}\right)\triangleq\int_{W^{\boldsymbol{1}}}\exp\left(-\Phi\left(V;\boldsymbol{y}\right)\right)\mu_{0}\left(dV\right)$
is the partition function. 
$\mu^{\boldsymbol{y}}$ is also called the posterior distribution for $V|\boldsymbol{y}$.

To sum up, we denote $\pi^{\boldsymbol{y}}$ as the posterior distribution
of $\mathfrak{q}|\boldsymbol{y}$ obtained from \ref{eq:formulation p(q|y)=00004Dint p(q|V) p(V|y)} and \ref{eq:1-level definition of =00005Cmu^y}, i.e.

\begin{equation}
\pi^{\boldsymbol{y}}\left(d\mathfrak{q}\right)=\int_{W^{\boldsymbol{1}}}\left[\mu^{\boldsymbol{y}}\left(dV\right)\pi^{V}\left(d\mathfrak{q}\right)\right].\label{eq:definition of pi^y}
\end{equation}

As we have assumed, the posterior distributions $ \mu^{\boldsymbol{y}} $ and $ \pi^V $ are absolutely continuous with respect to the prior distributions $ \mu_0 $ and $ \pi_0 $ correspondingly, so the absolute continuity of $ \pi^{\boldsymbol{y}} $ with respect to $ \pi_0 $ holds from Eqn. \ref{eq:definition of pi^y}. Moreover, we give the R-N derivative as follows without proof.
\begin{lem}
The R-N derivative of the posterior distribution $\pi^{\boldsymbol{y}}$ w.r.t. the prior distribution $\pi_{0}$ is
\begin{equation}
\dfrac{d\pi^{\boldsymbol{y}}}{d\pi_{0}}\left(\mathfrak{q}\right)=\int_{W^{\boldsymbol{1}}}\left[\dfrac{d\mu^{\boldsymbol{y}}}{d\mu_{0}}\left(V\right)\dfrac{d\pi^{V}}{d\pi_{0}}\left(\mathfrak{q}\right)\mu_{0}\left(dV\right)\right].\label{eq:formulation d pi^y/d =00005Cpi_0}
\end{equation}
\end{lem}

With these posterior distribution established, now we can examine the coherence between our model (sampling $\overline{O}\left[\mathfrak{q}\right]$ from $\pi^{\boldsymbol{y}}\left(d\mathfrak{q}\right)$) and algorithm (sampling $G^O\left(V\right)$ from $\mu^{\boldsymbol{y}}\left(dV\right)$) by the following proposition.

\begin{prop}
\label{prop:E mu-y G-O(V) = E pi-y O(q)}
The following two weighted averages are the same:
\[
	\mathbb{E}_{\mu^{\boldsymbol{y}}\left(dV\right)}G^{O}\left(V\right)=\mathbb{E}_{\pi^{\boldsymbol{y}}\left(d\mathfrak{q}\right)}\overline{O}\left[\mathfrak{q}\right].
\]
\end{prop}
\begin{proof}
The proof requires no more than a direct calculation: 
by definition of $\mu^{\boldsymbol{y}}\left(dV\right)$ and $G^O\left(V\right)$ we have
\[
	\mathbb{E}_{\mu^{\boldsymbol{y}}\left(dV\right)}G^{O}\left(V\right)
	= \int_{W^{\boldsymbol{1}}} \mu^{\boldsymbol{y}}\left(dV\right) \left[\int_X \pi^{V}\left(d\mathfrak{q}\right) \overline{O}\left[\mathfrak{q}\right]\right].
\]
On the other hand, by interchanging the order of integrals and definition of $\pi^{\boldsymbol{y}}\left(d\mathfrak{q}\right)$, we have
\begin{align*}
	\int_{W^{\boldsymbol{1}}} \mu^{\boldsymbol{y}}\left(dV\right) \left[\int_X \pi^{V}\left(d\mathfrak{q}\right) \overline{O}\left[\mathfrak{q}\right]\right]
	&= \int_X \overline{O}\left[\mathfrak{q}\right] \left[\int_{W^{\boldsymbol{1}}} \mu^{\boldsymbol{y}}\left(dV\right) \pi^{V}\left(d\mathfrak{q}\right)\right] \\
	&= \int_X \overline{O}\left[\mathfrak{q}\right] \pi^{\boldsymbol{y}}\left(d\mathfrak{q}\right) \\
	&= \mathbb{E}_{\pi^{\boldsymbol{y}}\left(d\mathfrak{q}\right)} \overline{O}\left[\mathfrak{q}\right].
\end{align*}
So we arrive at $\mathbb{E}_{\mu^{\boldsymbol{y}}\left(dV\right)}G^{O}\left(V\right)=\mathbb{E}_{\pi^{\boldsymbol{y}}\left(d\mathfrak{q}\right)} \overline{O}\left[\mathfrak{q}\right]$.
\end{proof}

At the end of this section, we write down the corresponding conclusion when $\boldsymbol{q}\in X_N=\mathbb{R}^N$ is a discrete ring under the ring polymer representation and $V\in W^{\boldsymbol{1}}_L$ is in the truncated potential space.
Denote $\mu_{0,L}$ as the marginal measure of $\mu_0$ on the subspace $W^{\boldsymbol{1}}_L$.
% marginal of $\mu^y$ is not the posterior distribution on the subspace
The conditional distribution of $V|\boldsymbol{y}$ on the subspace $W^{\boldsymbol{1}}_L$ can be similarly defined as
\[
\mu^{\boldsymbol{y}}_{L,N}\left(dV\right)\triangleq\mu_{0, L}\left(dV\right)\frac{1}{Z_{L,N}\left(\boldsymbol{y}\right)}\exp\left(-\Phi_N\left(V;\boldsymbol{y}\right)\right)\label{eq:1-level definition of mu^y_L}
\]
where $Z_{L,N}\left(\boldsymbol{y}\right)\triangleq \int_{W^{\boldsymbol{1}}_L}\mu_{0, L}\left(dV\right)\exp\left(-\Phi_N\left(V;\boldsymbol{y}\right)\right)$ is the partition function and
\[\Phi_N\left(V;\boldsymbol{y}\right)\triangleq\frac{1}{2}\left|\left|\Gamma_{\boldsymbol{\eta}}^{-\frac{1}{2}}\left(\boldsymbol{y}-G^{A}_N\left(V\right)\right)\right|\right|_2^{2}-\frac{1}{2}\left|\left|\Gamma_{\boldsymbol{\eta}}^{-\frac{1}{2}}\boldsymbol{y}\right|\right|_2^{2}\]
is the negative log likelihood with consideration of the ring polymer approximation.
Since the posterior distribution of $\boldsymbol{q}|\boldsymbol{y}$ depends both on the bead number $N$ and the truncation level $L$, we denote it as $\pi^{\boldsymbol{y}}_{L, N}$.
A similar computation gives
\[
\pi^{\boldsymbol{y}}_{L,N}\left(d\mathfrak{q}\right)=\int_{W^{\boldsymbol{1}}_L}\left[\mu^{\boldsymbol{y}}_{L,N}\left(dV\right)\pi^{V}_N\left(d\mathfrak{q}\right)\right].\label{eq:definition of pi^y_L}
\]
Two similar conclusions are listed below without proof.
\begin{lem}
	The R-N derivative of the posterior distribution $\pi^{\boldsymbol{y}}_N$ w.r.t. the prior distribution $\pi_{0, N}$ is
	\begin{equation}
	\dfrac{d\pi^{\boldsymbol{y}}_{L, N}}{d\pi_{0,N}}\left(\boldsymbol{q}\right)=\int_{W^{\boldsymbol{1}}_L}\left[\dfrac{d\mu^{\boldsymbol{y}}_{L,N}}{d\mu_{0, L}}\left(V\right)\dfrac{d\pi^{V}_N}{d\pi_{0, N}}\left(\mathfrak{q}\right)\mu_{0, L}\left(dV\right)\right].\label{eq:formulation dpi^y_N/dpi_0N}
	\end{equation}
\end{lem}
\begin{prop}
	\label{prop:E mu^y_L G-O(V) = E pi^y_N O(q)}
	The following two weighted averages are the same:
	\[
	\mathbb{E}_{\mu^{\boldsymbol{y}}_{L,N}\left(dV\right)}G^{O}_N\left(V\right)=\mathbb{E}_{\pi^{\boldsymbol{y}}_{L,N}\left(d\boldsymbol{q}\right)}\overline{O}\left[\boldsymbol{q}\right].
	\]
\end{prop}

\section{The Inversion Algorithm\label{sec:The-Inversion-Algorithm}}

In this section, we will introduce an inversion algorithm for the
thermal average sampling in the single level quantum system. Recall that our
interest lies in the conditional variable $\mathfrak{q}|\boldsymbol{y}$,
described by 
\[
p\left(\mathfrak{q}|\boldsymbol{y}\right)=\int_{W^{\boldsymbol{1}}}p\left(\mathfrak{q}|V\right)p\left(V|\boldsymbol{y}\right)dV
\]
in which the potential $V$ is viewed as an intermediate variable.
Therefore, by making use of the PIMD method (Sec. \ref{sec:A-Review-on-Forward-Problem})
to sample $\mathfrak{q}|V$ and the posterior distribution sampler
(Sec. \ref{subsec:Measure-preserving-dynamics} in Appendix) to sample
$V|\boldsymbol{y}$, the inversion algorithm can sample the conditional
variable $\mathfrak{q}|\boldsymbol{y}$ as well as the testing observation
$G^{O}\left(V\right)|\boldsymbol{y}$.

\subsection{Algorithm overview}

The algorithm is mainly two-staged: the first stage is to sample a
fixed potential function $\widehat V^{\left(k\right)}$ proposed by the prior
distribution and obtain the $k$-th training thermal average $\boldsymbol{y}^{\left(k\right)}=G^A\left(\widehat V^{(k)}\right)$
based on path integral molecular dynamics; at the second stage, the
average $\boldsymbol{y}^{\left(k\right)}$ is compared against the
ground truth $\boldsymbol{y}^{*}$ and previous sample $\boldsymbol{y}^{\left(k-1\right)}$,
where Metropolis-Hasting method is used to decide whether to keep
the proposed $\widehat V^{\left(k\right)}$ or not. By this way, the conditional
variable $V|\boldsymbol{y}^{*}$ can be sampled accurately up to the error in PIMD simulations.

\begin{algorithm}
\begin{algor}[1]
\item [{{*}}] Initialize the potential function $V^{\left(0\right)}=\widehat{V}^{\left(0\right)}$
(one choice is the harmonic potential $V_{o}$) and the noise covariance
matrix $\Gamma_{\boldsymbol{\eta}}$.
\item [{{*}}] Obtain the ground truth of training observation $\boldsymbol{y}^{*}$
(with or without noise).
\item [{{*}}] Set $k=0$.
\item [{while}] stopping condition is not satisfied (e.g. $k<K_{total}$)
\begin{algor}[1]
\item [{{*}}] Sample a sequence of bead configurations $\boldsymbol{q}_{1},\dots\boldsymbol{q}_{N_{1}}$
distributed according to the measure $\pi^{\widehat{V}^{\left(k\right)}}$
induced by the proposed potential function $\widehat{V}^{\left(k\right)}$.\label{enu:alg1-sample-q}
\item [{{*}}] Average the observations $\overline{A}\left(\boldsymbol{q}_{i}\right)$
to obtain the $k$-th training observation $\widehat{\boldsymbol{y}}^{\left(k\right)}=\frac{1}{N_{1}}\sum_{i=1}^{N_{1}}\overline{A}\left(\boldsymbol{q}_{i}\right)$.
(Note: $\boldsymbol{y}^{\left(0\right)}=\widehat{\boldsymbol{y}}^{\left(0\right)}$)
\item [{{*}}] Compute the negative log likelihood 
\[
\widehat{\Phi}^{\left(k\right)}\triangleq\Phi\left(\widehat{V}^{\left(k\right)};\boldsymbol{y}^{*}\right)=\left\langle \widehat{\boldsymbol{y}}^{\left(k\right)},\Gamma_{\boldsymbol{\eta}}^{-1}\left(\frac{1}{2}\widehat{\boldsymbol{y}}^{\left(k\right)}-\boldsymbol{y}^{*}\right)\right\rangle .
\]
\item [{if}] $k>0$
\begin{algor}[1]
\item [{{*}}] Pick a random number $r^{\left(k\right)}\sim\mathcal{U}\left[0,1\right]$.
\item [{if}] $r^{\left(k\right)}<\exp\left(\min\left[0,\Phi^{\left(k-1\right)}-\widehat{\Phi}^{\left(k\right)}\right]\right)$\label{enu:alg1-mc-method}
\begin{algor}[1]
\item [{{*}}] Accept the proposal: $V^{\left(k\right)}=\widehat{V}^{\left(k\right)}$,
$\boldsymbol{y}^{\left(k\right)}=\widehat{\boldsymbol{y}}^{\left(k\right)}$
and $\Phi^{\left(k\right)}=\widehat{\Phi}^{\left(k\right)}$.
\end{algor}
\item [{else}]~
\begin{algor}[1]
\item [{{*}}] Reject the proposal: $V^{\left(k\right)}=V^{\left(k-1\right)}$,
$\boldsymbol{y}^{\left(k\right)}=\boldsymbol{y}^{\left(k-1\right)}$
and $\Phi^{\left(k\right)}=\Phi^{\left(k-1\right)}$.
\end{algor}
\item [{endif}]~
\end{algor}
\item [{endif}]~
\item [{{*}}] Store the decided potential function $V^{\left(k\right)}$.
\item [{{*}}] Obtain thermal average $G^{O}\left(V^{\left(k\right)}\right)$
of test observables if needed.
\item [{{*}}] Propose $\widehat{V}^{\left(k+1\right)}$ based on $V^{\left(k\right)}$, e.g. by gaussian random walk with parameter $\rho$.\label{enu:alg1-proposal}
\end{algor}
\item [{endwhile}]~
\end{algor}
\caption{Inversion Sampler For 1-Level Quantum Thermal Average Problem}

\label{alg:1-level algorithm}
\end{algorithm}

\subsection{Algorithm implementation}

\subsubsection{PIMD solver (Line \ref{enu:alg1-sample-q} in Alg. \ref{alg:1-level algorithm})}

To efficiently solve the forward problem, we implement a PIMD solver
in C++ and perform tests
on a machine equipped with a Intel Core i5-7300HQ
(single-threaded at 3.5GHz).

\subsubsection{Metropolis Hasting method (Line \ref{enu:alg1-mc-method} in Alg.
\ref{alg:1-level algorithm})}

The measure of interest $\mu^{\boldsymbol{y}}$ is absolutely continuous
with respect to $\mu_{0}$ and its R-N derivative is in proportion to $\exp\left(-\Phi\left(V;\boldsymbol{y}\right)\right)$.
So in order to sample $\mu^{\boldsymbol{y}}$ based on proposals from
$\mu_{0}$, according to Assu. \ref{assu:mcmc-reversible-coondition},
the decision function can be chosen as
\[
a\left(V_{old}\to V_{new}\right)=\exp\left(\min\left\{ 0,\Phi\left(V_{old}\right)-\Phi\left(V_{new}\right)\right\} \right)
\]
so that the detailed balance condition will be satisfied as long as
the proposal kernel is reversible with respect to the prior distribution.

\subsubsection{Proposal kernel of potential function (Line \ref{enu:alg1-proposal}
in Alg. \ref{alg:1-level algorithm})}

The next step is to choose a proposal kernel which is reversible with respect to
the prior distribution. Recall that each potential function $V$ in
the truncated space $W_{L}^{\boldsymbol{1}}$ can be expressed as
\[
V=V_{o}+\sum_{i=0}^{L}\xi_{i}\gamma_{i}\phi_{i}
\]
while the covariance operator (defined in Eqn. \ref{eq:definition truncated covariance})
has the matrix form $\boldsymbol{diag}\left(\gamma_{0}^{2},\dots\gamma_{L}^{2}\right)$,
so to propose $V$ is essentially to propose $\left\{ \xi_{i}\right\} _{i=0}^{L}$.

To avoid high rejection rate, we adopt the gaussian random walk (in \cite{Cotter2012} also called ``pre-conditioned CN proposal'') in the proposal stage: given a gaussian random variable $r$ and parameter $ 0<\rho<1$, then by generating an independent standard gaussian variable $g \sim \mathcal{N}\left(0, 1\right)$ we obtain
\[
r^* = \rho r + \sqrt{1-\rho^2} g
\]
which shares the same mean and variance with $r$. Moreover, the correlation between $r^*$ and $r$ is exactly $\rho$, so a larger $\rho$ can make two successive proposals closer while keeping the desired invariant distribution unchanged.

In the 2-level problem we also need to propose r.v. that follow the exponential distribution. Based on the observation that the sum of squares of two standard gaussian variables follows the exponential distribution, we can prepare two standard gaussian variable $g_1,g_2\sim\mathcal{N}\left(0,1\right)$ and use the given exponential variable $r$ and parameter $0<\rho<1$ to obtain
\[
r^*=\left(\rho\sqrt{r}+\bar{\rho}g_{1}\right)^{2}+\left(\bar{\rho}g_{2}\right)^{2}, \bar{\rho} \triangleq \sqrt{\frac{1-\rho^2}2},
\]
which also shares the same mean and variance with $r$ and has a $\rho$ correlation with $r$.

\iffalse
% tautology on proposal
However, a direct random uniform proposal regardless of previous samples
will lead to unaffordable rejections rates, which is often caused by
the small scale of the noise covariance operator $\Gamma_{\boldsymbol{\eta}}$. To lower the rejection rate, we can take the
advantage of the fact that the prior distribution of interest is Gaussian
and implement a simple but efficient proposal generator, which ensures
that two successive proposals are close to keep the negative log potential
$\Phi$ flowing in a smoother way and also keeps the invariant distribution
unchanged. For the algorithms in detail, we refer Section \ref{sec:Reversible-Proposals}
in Appendix and \cite{Cotter2012} for the readers. 
\fi

\subsubsection{Regularity of $V\in W^{\boldsymbol{1}}$}
Recall that elements in $W^{\boldsymbol{1}}$ can be expressed in the
form of 
\begin{equation}
V=V_{o}+\sum_{i=0}^{\infty}\xi_{i}\gamma_{i}\phi_{i}\label{eq:potential-decomposition}
\end{equation}
where $\left\lbrace \xi_i\right\rbrace $ are the random variables and $ \left\lbrace \phi_i \right\rbrace $ is the basis.
In particular, the decaying sequence $ \left\lbrace \gamma_i \right\rbrace $ serves as a parameter controlling the regularity of $V$. 
We assume that, given a constant parameter $\beta\in \left(0,1\right)$, the asymptotic behavior of $\left\lbrace \gamma_i \right\rbrace$ is given by
\begin{equation}
\gamma_i = \mathcal{O}\left(i^{-s}\right), s>\max\left\lbrace 1,\frac{\beta+2}{4\left(1-\beta\right)} \right\rbrace.
\end{equation}

Under such assumption, we have the following properties:

\begin{enumerate}
	\item The exponent $s$ is larger than 1, so by Prop. \ref{prop:function-series-converge-in-L-infty}, $V-V_o$ is in $L^\infty \cap L^2$ and $V$ is in $W^{\boldsymbol{1}}$.
	\item The exponent $s$ is larger than $\frac{\beta+2}{4\left(1-\beta\right)}$, so by Prop. \ref{prop:holder-continuity-of-function-series}, $V$ has $\beta$-order H\"{o}lder continuity.
\end{enumerate}

\section{Stability Analysis\label{sec:Prior-Analysis}}
 
In terms of stability, recall that the variable of interest
is $\mathfrak{q}|\boldsymbol{y}$ in distribution of $\pi^{\boldsymbol{y}}$,
so we want to establish a continuous dependency of $\pi^{\boldsymbol{y}}$
on $\boldsymbol{y}$. Targeted at this, the proof is structured in
the following manner. First, Lem. \ref{lem:1-level d(pi^V)<d(V)}
shows that the induced measure $\pi^{V}$ on configuration space continuously
depends on the potential $V$, based on which Cor. \ref{cor:1layer-proof-V-bound-avgA}
ensures that the thermal average $G^{A}\left(V\right)$ is a continuous
with respect to $V$. Then Lem. \ref{lem:1-level d(=00005Cmu^y)<d(y)}
shows that the induced measure $\mu^{\boldsymbol{y}}$ on potential
space continuously depends on the training observation $\boldsymbol{y}$.
The main theorem (Thm. \ref{thm:1-level d(=00005Cpi^y)<d(y) (main result)})
utilizes the two lemmas and the formula for $\pi^{\boldsymbol{y}}$
to control the perturbation between $\pi^{\boldsymbol{y}}$ by that
between $\mu^{\boldsymbol{y}}$, leading to the conclusion. Finally,
Cor. \ref{cor:1-level d(E(A'))<d(y)} shows how to bound the perturbation
in testing observations by that in training observations.
\begin{notation*}
Given a separable Banach space $\mathcal{X}$ and two functions $f_{1},f_{2}$
on $\mathcal{X}\times\mathcal{X}$. We call $f_{1}\apprle_{r}f_{2}$
if, for every fixed $r>0$, there is $C=C\left(r\right)>0$ such that
\[
f_{1}\left(x,x'\right)\leq Cf_{2}\left(x,x'\right),\forall x,x'\in B_{\mathcal{X}}\left(0,r\right).
\]
\end{notation*}
\begin{lem}
\label{lem:1-level d(pi^V)<d(V)}Given two potential functions $V_{1},V_{2}\in W^{\boldsymbol{1}}$,
we have
\[
d_{Hell}\left(\pi^{V_{1}},\pi^{V_{2}}\right)\lesssim_{r}\left|\left|V_{1}-V_{2}\right|\right|_{W^{\boldsymbol{1}}}
\]
where the definition of Hellinger distance is given by Def. \ref{def:Hellinger-distance}
in the Appendix.
\end{lem}

The proof of this lemma is a direct application of Thm. \ref{thm:well-posedness}
in the Appendix.
\begin{proof}
First we check the conditions for Assu. \ref{assu:phi-regularity}:
notice that $\Phi\left(\mathfrak{q};V\right)=\int_{0}^{\beta}\mathring{V}\circ\mathfrak{q}\,d\tau$
is continuous as a function of $\mathfrak{q}$, and that
\[
\Phi\left(\mathfrak{q};V\right)\ge-\beta\left|\left|\mathring{V}\right|\right|_{L^{\infty}}\ge-\beta\left|\left|V\right|\right|_{W^{\boldsymbol{1}}},
\]
\[
\left|\Phi\left(\mathfrak{q};V_{1}\right)-\Phi\left(\mathfrak{q};V_{2}\right)\right|=\left|\int_{0}^{\beta}\left(V_{1}-V_{2}\right)\circ\mathfrak{q}\,d\tau\right|\le\beta\left|\left|V_{1}-V_{2}\right|\right|_{\infty}\le\beta\left|\left|V_{1}-V_{2}\right|\right|_{W^{\boldsymbol{1}}},
\]
thus we can take $M_{1}\left(r\right)=\beta r$ and $M_{2}=\beta$.

Since $M_{1}$ and $M_{2}$ are constant with respect to $\mathfrak{q}$, by
Thm. \ref{thm:well-posedness}, there exists $C=C\left(r\right)$
such that for all $V_{1},V_{2}\in B_{W^{\boldsymbol{1}}}\left(0,r\right)$,
\[
d_{Hell}\left(\pi^{V_{1}},\pi^{V_{2}}\right)\le C\left|\left|V_{1}-V_{2}\right|\right|_{W^{\boldsymbol{1}}}.
\]
\end{proof}
\begin{cor}
\label{cor:1layer-proof-V-bound-avgA}The ensemble average of a given
bounded observable $A$ is continuous as a function of the potential
function, i.e.
\[
\left|G^{A}\left(V_{1}\right)-G^{A}\left(V_{2}\right)\right|\lesssim_{r}\left|\left|V_{1}-V_{2}\right|\right|_{W^{\boldsymbol{1}}}
\]
where $G^{A}\left(V\right)\triangleq\mathbb{E}_{\pi^{V}\left(d\mathfrak{q}\right)}\frac{1}{\beta}\int_{0}^{\beta}A\circ\mathfrak{q}\,d\tau=\mathbb{E}_{\pi^{V}\left(d\mathfrak{q}\right)}\overline{A}\left[\mathfrak{q}\right]$.
\end{cor}

\begin{proof}
This is a direct corollary of Lem. \ref{lem:continuity-linfty-observable}
and Lem. \ref{lem:1-level d(pi^V)<d(V)} since $G^{A}\left(V\right) $
is a bounded function on $W^{\boldsymbol{1}}$.
\end{proof}
\begin{lem}
\label{lem:1-level d(=00005Cmu^y)<d(y)}Given two training observations
$\boldsymbol{y}_{1},\boldsymbol{y}_{2}$ and assume the training observable
$A$ is bounded, we have
\[
d_{Hell}\left(\mu^{\boldsymbol{y}_{1}},\mu^{\boldsymbol{y}_{2}}\right)\lesssim_{r}\left|\left|\boldsymbol{y}_{1}-\boldsymbol{y}_{2}\right|\right|_{2}
\]
where $\left|\left|\cdot\right|\right|_{2}$ is the Euclidean norm
on $\mathbb{R}^{N_{T}}$ and the Hellinger distance is given by Def.
\ref{def:Hellinger-distance}.
\end{lem}

The proof of this lemma is another direct application of Thm. \ref{thm:well-posedness}.
\begin{proof}
Again we check the conditions for Assu. \ref{assu:phi-regularity}:
notice that
\[
\Phi\left(V;\boldsymbol{y}\right)=\frac{1}{2}\left|\left|\Gamma_{\boldsymbol{\eta}}^{-\frac{1}{2}}\left(\boldsymbol{y}-G^{A}\left(V\right)\right)\right|\right|_2^{2}-\frac{1}{2}\left|\left|\Gamma_{\boldsymbol{\eta}}^{-\frac{1}{2}}\boldsymbol{y}\right|\right|_2^{2}
\]
is continuous as a function of the two variables $\left(V,\boldsymbol{y}\right)$
(by Cor. \ref{cor:1layer-proof-V-bound-avgA}), and that
\[
\Phi\left(V;\boldsymbol{y}\right)\ge-\frac{1}{2}\left|\left|\Gamma_{\boldsymbol{\eta}}^{-\frac{1}{2}}\boldsymbol{y}\right|\right|_{2}^{2},
\]
\[
\left|\Phi\left(V;\boldsymbol{y}_{1}\right)-\Phi\left(V;\boldsymbol{y}_{2}\right)\right|\le\left[\frac{1}{2}\left|\left|\boldsymbol{y}_{1}+\boldsymbol{y}_{2}\right|\right|_{2}+\left|\left|A\right|\right|_{\infty}\right]\left|\left|\Gamma_{\boldsymbol{\eta}}^{-1}\right|\right|_{2}\left|\left|\boldsymbol{y}_{1}-\boldsymbol{y}_{2}\right|\right|_{2},
\]
we can take $M_{1}=\frac{1}{2}\left|\left|\Gamma_{\boldsymbol{\eta}}^{-\frac{1}{2}}\boldsymbol{y}\right|\right|_{2}^{2}$
and $M_{2}\left(r\right)=\left|\left|\Gamma_{\boldsymbol{\eta}}^{-1}\right|\right|_{2}\left(r+\left|\left|A\right|\right|_{\infty}\right)$,
where $\left|\left|\Gamma_{\boldsymbol{\eta}}^{-1}\right|\right|_{2}$ is the operator
norm induced by the vector norm $\left|\left|\cdot\right|\right|_{2}$.
Since they are also constant with respect to $V$, by Thm. \ref{thm:well-posedness},
there is $C=C\left(r\right)$ such that for all $\boldsymbol{y}_{1},\boldsymbol{y}_{2}\in B_{Y}\left(0,r\right)$,
\[
d_{Hell}\left(\mu^{\boldsymbol{y}_{1}},\mu^{\boldsymbol{y}_{2}}\right)\le C\left|\left|\boldsymbol{y}_{1}-\boldsymbol{y}_{2}\right|\right|_{2}.
\]
\end{proof}
\begin{thm}
\label{thm:1-level d(=00005Cpi^y)<d(y) (main result)}Given two training
observations $\boldsymbol{y}_{1},\boldsymbol{y}_{2}$ and assume the
training observable $A$ is bounded, the distance between two induced
measure on $\mathcal{L}\mathbb{R}$ can be bounded by the their distance,
i.e.
\[
d_{TV}\left(\pi^{\boldsymbol{y}_{1}},\pi^{\boldsymbol{y}_{2}}\right)\lesssim_{r}\left|\left|\boldsymbol{y}_{1}-\boldsymbol{y}_{2}\right|\right|_{2}.
\]
\end{thm}

\begin{proof}
The proof consists of two steps: first we will bound the distance
between $\pi^{\boldsymbol{y}}$ by the distance between $\mu^{\boldsymbol{y}}$,
then we will compare the latter part with the distance between $\boldsymbol{y}$.

By definition of TV distance and the fact that $\pi^{\boldsymbol{y}}\ll\pi_{0}$,
we can expand the left-hand side as
\begin{align*}
d_{TV}\left(\pi^{\boldsymbol{y}_{1}},\pi^{\boldsymbol{y}_{2}}\right) & =\int_{\mathcal{L}\mathbb{R}}\pi_{0}\left(d\mathfrak{q}\right)\left|\dfrac{d\pi^{\boldsymbol{y}_{1}}}{d\pi_{0}}\left(\mathfrak{q}\right)-\dfrac{d\pi^{\boldsymbol{y}_{2}}}{d\pi_{0}}\left(\mathfrak{q}\right)\right|\\
 & =\int_{\mathcal{L}\mathbb{R}}\pi_{0}\left(d\mathfrak{q}\right)\left|\int_{W^{\boldsymbol{1}}}\mu_{0}\left(dV\right)\thinspace\dfrac{d\pi^{V}}{d\pi_{0}}\left(\mathfrak{q}\right)\thinspace\left[\dfrac{d\mu^{\boldsymbol{y}_{1}}}{d\mu_{0}}\left(V\right)-\dfrac{d\mu^{\boldsymbol{y}_{2}}}{d\mu_{0}}\left(V\right)\right]\right|,
\end{align*}
where the second equation results from the definition of $\pi^{\boldsymbol{y}}$
(Eqn. \ref{eq:formulation d pi^y/d =00005Cpi_0}). By taking the absolute
sign into the second integral and a change to the order of integral (by Tonelli's theorem),
we have
\begin{align*}
d_{TV}\left(\pi^{\boldsymbol{y}_{1}},\pi^{\boldsymbol{y}_{2}}\right) & \le\int_{\mathcal{L}\mathbb{R}}\pi_{0}\left(d\mathfrak{q}\right)\left[\int_{W^{\boldsymbol{1}}}\mu_{0}\left(dV\right)\thinspace\dfrac{d\pi^{V}}{d\pi_{0}}\left(\mathfrak{q}\right)\thinspace\left|\dfrac{d\mu^{\boldsymbol{y}_{1}}}{d\mu_{0}}\left(V\right)-\dfrac{d\mu^{\boldsymbol{y}_{2}}}{d\mu_{0}}\left(V\right)\right|\right]\\
 & =\int_{W^{\boldsymbol{1}}}\mu_{0}\left(dV\right)\left[\int_{\mathcal{L}\mathbb{R}}\pi_{0}\left(d\mathfrak{q}\right)\thinspace\dfrac{d\pi^{V}}{d\pi_{0}}\left(\mathfrak{q}\right)\thinspace\left|\dfrac{d\mu^{\boldsymbol{y}_{1}}}{d\mu_{0}}\left(V\right)-\dfrac{d\mu^{\boldsymbol{y}_{2}}}{d\mu_{0}}\left(V\right)\right|\right].
\end{align*}
Notice that the R-N derivative $\dfrac{d\mu^{\boldsymbol{y}}}{d\mu_{0}}\left(V\right)$
does not depend on $\mathfrak{q}$, we can take the term outside the
second integral:
\[
d_{TV}\left(\pi^{\boldsymbol{y}_{1}},\pi^{\boldsymbol{y}_{2}}\right)\le\int_{W^{\boldsymbol{1}}}\left\{ \mu_{0}\left(dV\right)\left|\dfrac{d\mu^{\boldsymbol{y}_{1}}}{d\mu_{0}}\left(V\right)-\dfrac{d\mu^{\boldsymbol{y}_{2}}}{d\mu_{0}}\left(V\right)\right|\left[\int_{\mathcal{L}\mathbb{R}}\pi_{0}\left(d\mathfrak{q}\right)\thinspace\dfrac{d\pi^{V}}{d\pi_{0}}\left(\mathfrak{q}\right)\right]\right\} .
\]
Now we can directly calculate the second integral since $\int_{\mathcal{L}\mathbb{R}}\pi_{0}\left(d\mathfrak{q}\right)\thinspace\dfrac{d\pi^{V}}{d\pi_{0}}\left(\mathfrak{q}\right)=\int_{\mathcal{L}\mathbb{R}}\pi^{V}\left(d\mathfrak{q}\right)=1$,
so again by the definition of TV distance
\begin{align*}
d_{TV}\left(\pi^{\boldsymbol{y}_{1}},\pi^{\boldsymbol{y}_{2}}\right) & \le\int_{W^{\boldsymbol{1}}}\mu_{0}\left(dV\right)\left|\dfrac{d\mu^{\boldsymbol{y}_{1}}}{d\mu_{0}}\left(V\right)-\dfrac{d\mu^{\boldsymbol{y}_{2}}}{d\mu_{0}}\left(V\right)\right|\\
 & =d_{TV}\left(\mu^{\boldsymbol{y}_{1}},\mu^{\boldsymbol{y}_{2}}\right).
\end{align*}

By Lem. \ref{lem:tv-hellinger-inequality}, we can bound the TV distance
by the Hellinger distance, leading to
\begin{align*}
d_{TV}\left(\pi^{\boldsymbol{y}_{1}},\pi^{\boldsymbol{y}_{2}}\right)\le d_{TV}\left(\mu^{\boldsymbol{y}_{1}},\mu^{\boldsymbol{y}_{2}}\right) & \le\sqrt{2}d_{Hell}\left(\mu^{\boldsymbol{y}_{1}},\mu^{\boldsymbol{y}_{2}}\right)\\
 & \lesssim_{r}\left|\left|\boldsymbol{y}_{1}-\boldsymbol{y}_{2}\right|\right|_{2},
\end{align*}
where the second inequality holds as a result of Lem. \ref{lem:1-level d(=00005Cmu^y)<d(y)}.
\end{proof}
\begin{cor}
\label{cor:1-level d(E(A'))<d(y)}Given two training observations
$\boldsymbol{y}_{1},\boldsymbol{y}_{2}$ and assume the training and
testing observables $A,O$ are bounded, the difference between the
expectations of the testing observable w.r.t the two posterior distributions
can be bounded by the distance of training observations:
\[
\left|\mathbb{E}_{\pi^{\boldsymbol{y}_{1}}}\overline{O}\left[\mathfrak{q}\right]-\mathbb{E}_{\pi^{\boldsymbol{y}_{2}}}\overline{O}\left[\mathfrak{q}\right]\right|\lesssim_{r}\left|\left|\boldsymbol{y}_{1}-\boldsymbol{y}_{2}\right|\right|_{2}.
\]
\end{cor}

\begin{proof}
This is a direct corollary of Lem. \ref{lem:continuity-linfty-observable}
and Thm. \ref{thm:1-level d(=00005Cpi^y)<d(y) (main result)} since
$\overline{O}\left[\mathfrak{q}\right]=\frac{1}{\beta}\int_0^\beta O\left(\mathfrak{q}\left(\tau\right)\right)\,d\tau$ is a bounded function on $X$.
\end{proof}

When we apply those estimates on the proposed algorithm, the bead number $N$ and the truncation level $L$ are finite, so we have to give the corresponding conclusions. However, most of the proof are the same since the only difference is the underlying measure space. Thus we list the conclusions as follows without proof:
\begin{thm}
	\label{thm:1-level d(pi^y_L,N)<d(y) (main result)}Assume the bead number $N$ and the truncation level $L$ are given and finite. Given two training
	observations $\boldsymbol{y}_{1},\boldsymbol{y}_{2}$ and assume the
	training observable $A$ is bounded, the distance between two induced
	measure on $\mathbb{R}^N$ can be bounded by the their distance,
	i.e.
	\[
	d_{TV}\left(\pi^{\boldsymbol{y}_{1}}_{L, N},\pi^{\boldsymbol{y}_{2}}_{L, N}\right)\lesssim_{r}\left|\left|\boldsymbol{y}_{1}-\boldsymbol{y}_{2}\right|\right|_{2}.
	\]
\end{thm}

\begin{rem*}
Normally, one expects convergence $\pi^{\boldsymbol{y}}_{L, N}\to \pi^{\boldsymbol{y}}_{L, \infty}$ with respect to $N$ (or even $\to \pi^{\boldsymbol{y}}$ when both $L$ and $N$ go to infinity), but it is not entirely clear in this case, where the approximation property relies on the rigorous justification of the path integral formulation. Whereas, if there exists a proper norm $\left|\left|\cdot\right|\right|_*$ on the configuration space $X=\mathcal{L}\mathbb{R}$ s.t. the following asymptotic estimation holds for every potential function $V\in W^{\boldsymbol{1}}$
\[
\left|S\left[\mathfrak{q}\right]-S_N\left(R_N\left(\mathfrak{q}\right)\right)\right|=\left|\int_0^\beta \mathring{V} \circ \mathfrak{q}\,\text{d}\tau -\beta_N \sum_{i=1}^N V\left(\mathfrak{q}\left(i\beta_N\right)\right)\right|\le M\left(\left|\left|\mathfrak{q}\right|\right|_*\right)\psi\left(N\right),
\]
where the $R_N:\mathcal{L}\mathbb{R}\to \mathbb{R}^N, \left(R_N\left(\mathfrak{q}\right)\right)_i=\mathfrak{q}\left(i\beta_N\right)$ is the restriction operator, $M$ is a non-decreasing function and $\lim_{N\to\infty}\psi\left(N\right)=0$,
then by Thm 4.9 in \cite{Dashti2017}, the following convergence holds
\[
d_\text{Hell}\left(\pi^V_N, \pi^V\right) \apprle_{r} \psi\left(N\right).
\]
Based on this we can expect $\pi^{\boldsymbol{y}}_{L, N}\to \pi^{\boldsymbol{y}}_{L, \infty}$ by estimating the error of R-N derivatives $\dfrac{\text{d}\mu^{\boldsymbol{y}}_{L,N}}{\text{d}\mu_{0, L}}$ and $\dfrac{\text{d}\pi^{V}_{N}}{\text{d}\pi_{0, N}}$ when $N\to\infty$.
\end{rem*}

\begin{cor}
	\label{cor:1-level d(E_N,L(O))<d(y)}Assume the bead number $N$ and the truncation level $L$ are given and finite. Given two training observations
	$\boldsymbol{y}_{1},\boldsymbol{y}_{2}$ and assume the training and
	testing observables $A,O$ are bounded, the difference between the
	expectations of the testing observable w.r.t the two posterior distributions
	can be bounded by the distance of training observations:
	\[
	\left|\mathbb{E}_{\pi^{\boldsymbol{y}_{1}}_{L, N}}\overline{O}\left(\boldsymbol{q}\right)-\mathbb{E}_{\pi^{\boldsymbol{y}_{2}}_{L, N}}\overline{O}\left(\boldsymbol{q}\right)\right|\lesssim_{r}\left|\left|\boldsymbol{y}_{1}-\boldsymbol{y}_{2}\right|\right|_{2}.
	\]
\end{cor}

\begin{rem*}
Similarly we can expect the convergence $\mathbb{E}_{\pi^{\boldsymbol{y}_{1}}_{L, N}}\overline{O}\left(R_N\left(\mathfrak{q}\right)\right)\to \mathbb{E}_{\pi^{\boldsymbol{y}}_{L, \infty}}\overline{O}\left[\mathfrak{q}\right]$ with respect to $N$ (or even $\to \mathbb{E}_{\pi^{\boldsymbol{y}}}\overline{O}\left[\mathfrak{q}\right]$) under certain assumptions made above, which means the test prediction under ring polymer representation is consistent with that under continuum limit.
We shall revisit this corollary on stability in the numerical study
section.
\end{rem*}

\section{Numerical tests (1 level system) \label{sec:Numerical-Study-on-1-Level-System}}

\subsection{Proof of concept\label{sec:1-level Proof-of-concept}}

First we will demonstrate that the algorithm works in a one-dimensional
system setting if the training observation $\boldsymbol{y^{*}}$ is
assumed without noise. The system is set-up as follows:
\begin{itemize}
\item Ground-truth potential function $V_{truth}=\frac{1}{2}x^{2}+5\sin\left(\frac{5x}{\pi}\right)\exp\left(-\frac{x^{2}}{2}\right)$,
shown in Fig. \ref{fig:1-level show-case illustration-potential}.
%% Ground truth  Green N500
%% Harmonic      Brown N500 30%
\item Training observable $A_{i}=e^{-\left(x-x_{i}\right)^{2}}$
where $x_{i}=\frac{i}{2}-2,i=0\dots8$, %% BlueGrey N100-N900
i.e. a series of Gaussians along the most probable locations.
\item Testing observable
$
O_{1}=\phi_{1}\left(2x\right),       %% Blue   N500
O_{2}=e^{-\left(x+1.25\right)^{2}},  %% Red    N500
O_{3}=e^{-\left(x-0.25\right)^{2}},  %% Purple N200 30%
O_{4}=\phi_{2}\left(3x\right),       %% Teal   N200 30%
O_{5}=\phi_{3}\left(3x\right)        %% Amber  N400 30%
$.
\item The noise covariance matrix $\Gamma_{\boldsymbol{\eta}}=\widetilde{\Gamma_{\eta}}I^{N_{1}\times N_{1}}$ with $\widetilde{\Gamma_{\eta}}=10^{-3}$
\item The potential covariance sequence $\gamma_{j}=4\cdot j^{-1.2}\left(j=0\dots12\right).$
\item Particle mass $M=10$, truncation level $L=12$ and bead number $N=16$.
\end{itemize}
\begin{figure}[!htb]
\includegraphics[width=1\linewidth]{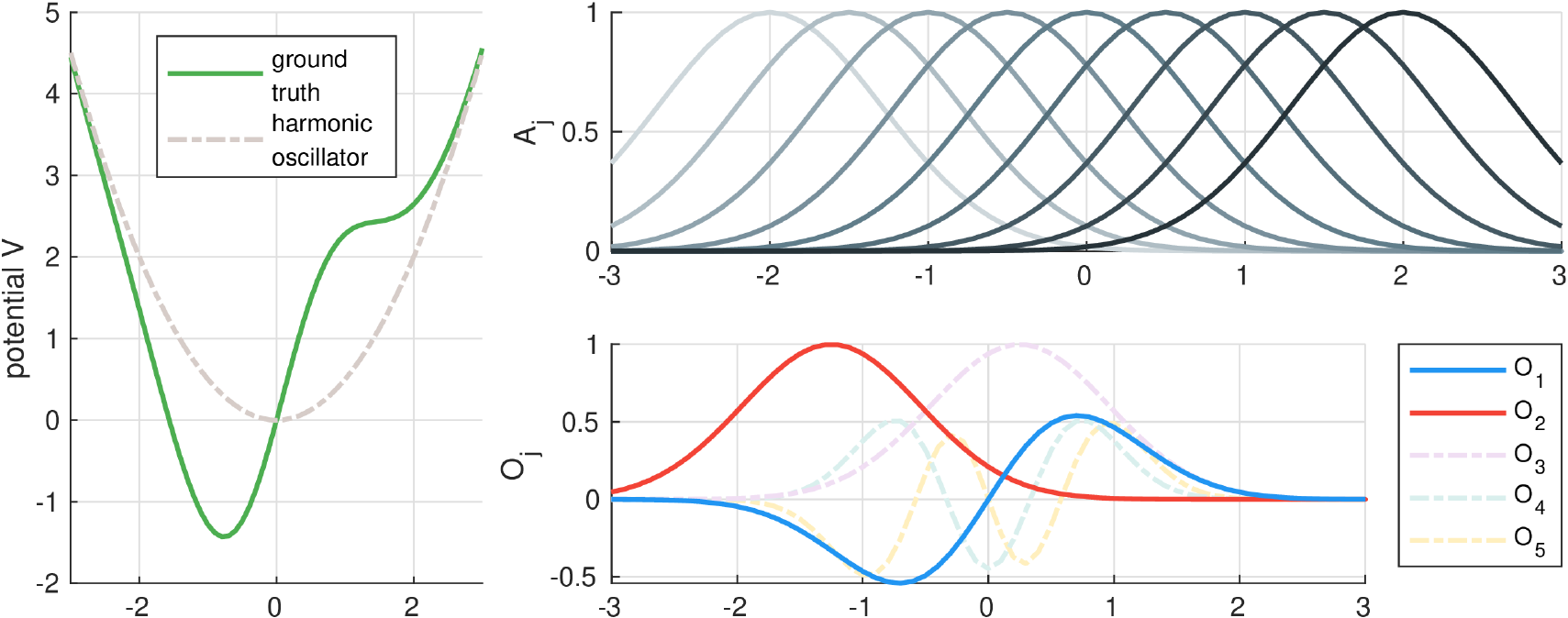}
\caption{Left figure: the ground truth potential $V_{truth}$ (green line) is lower 
near $x=-1$ and higher near $x=1$ as compared to the
harmonic oscillation potential (brown faded line, also used as the initial
potential $V^{\left(0\right)}$). \\
Top-right figure: illustration for
training observables $\left\lbrace A_j \right\rbrace_{j=0}^8$,
each of which is a Gaussian function distributed along the
most probable locations. \\
Bottom-right figure: illustration for testing
observables $\left\lbrace O_j \right\rbrace_{j=1}^5$,
varying in the size of support set and oscillation intensity.
Among the five testing observables, the first two will be selected for demonstration
since the rest three give similar results.}
\label{fig:1-level show-case illustration-potential}
\end{figure}

We ran 10 independent runs in total with 1600 proposals in each run.
We select two of the five testing observables to report since the results
of the rest three are similar.
In the following figures, we will always label the initial guess (obtained at the
first iteration, i.e. by setting the harmonic oscillation potential
$V_{o}=\frac{1}{2}x^{2}$) as the brown dotted line and the ground truth
(obtained by setting ground truth $V_{truth}$)
as the green dashed line; we also indicate twice the standard error by
the shaded area.

To ensure the algorithm indeed samples the entire landscape of the posterior distribution, we provide some plots on the statistic characteristics, where Fig. \ref{fig:1-level show-case acceptance_rate_mse} shows the acceptance rate and the decrease in mean squared error and Fig. \ref{fig:1-level show-case auto_correlation} shows the auto-correlation function for test observables and the coordinates $\xi_i$ of the sampled potential $V^{\left(k\right)}$. Since we use random gaussian walk in the proposal of $V^{\left(k\right)}$ and by default $\rho=0.95$, so we expect a large auto-correlation time. We fix $t_{ac}=50$ in the following calculations, which means we have $1600/50\times10=320$ independent samples in total.

\begin{figure}[!htb]
\includegraphics[width=1\linewidth]{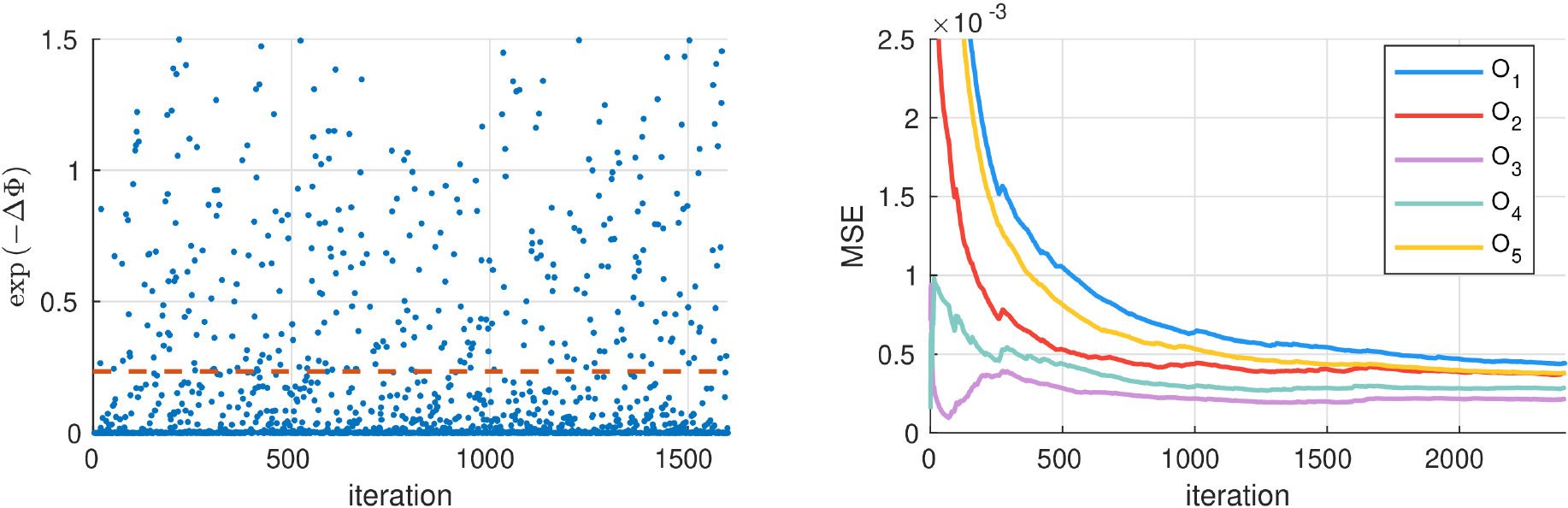}
	\caption{
		Left figure: transition probability ratio $\exp\left(-\Delta \Phi\right)=\exp\left(\Phi^{\left(k-1\right)}-\widehat{\Phi}^{\left(k\right)}\right)$ is plotted in blue dots and the mean acceptance rate is indicated by the red dashed line. Right figure: mean squared sampling error in test observables. The mean error stabilizes after the 1600th iteration. In both figures, the horizontal axis corresponds to the iteration of the sampling procedure.
	}
	\label{fig:1-level show-case acceptance_rate_mse}
\end{figure}

\begin{figure}[!htb]
\includegraphics[width=1\linewidth]{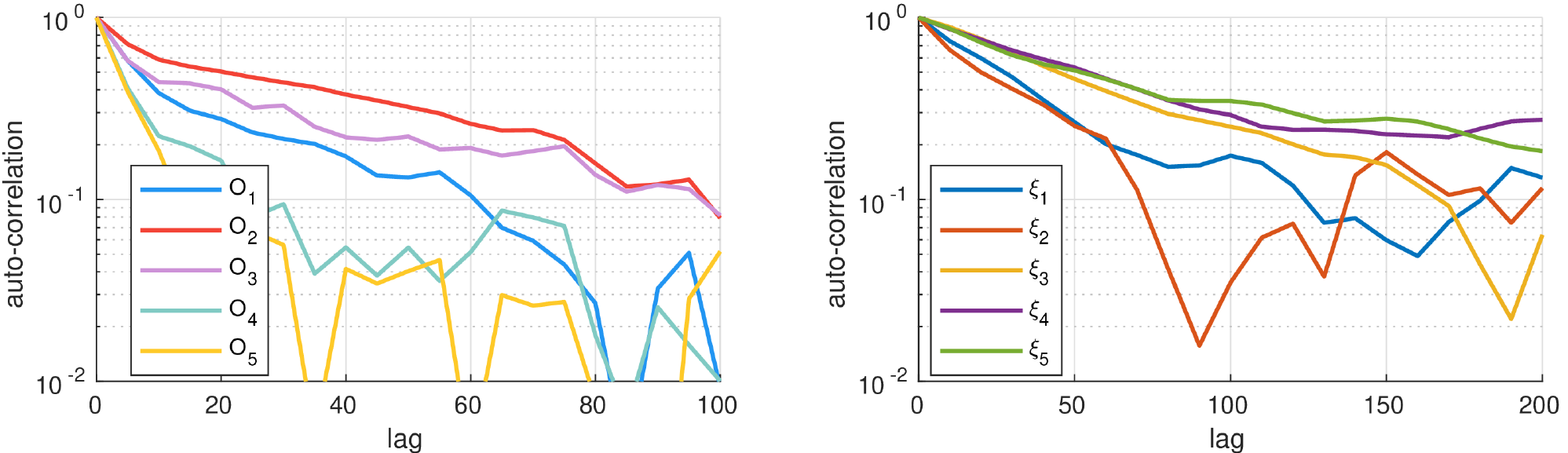}
	\caption{
		Auto-correlation in test observables (left figure) and in coordinates of the sampled potential (right figure). The horizontal axis corresponds to the lag between samples. When the lag reaches 50, the auto-correlation gets below 0.5 in every variable of both figures.
	}
	\label{fig:1-level show-case auto_correlation}
\end{figure}

The results of averaged sampled test observables,  potential and density function are shown in Fig. \ref{fig:1-level show-case averaged-observables}
and \ref{fig:1-level show-case averaged-potential, density}.
Judging from the figures, the potential and density landscapes can correctly sampled where the probability density is relatively large. As for the thermal average of the testing observables, the averaged samples will converge towards the ground truth and the standard error will decrease over iterations. We will further discuss the residual between the convergent value and the ground truth in the following sections.

\begin{figure}[!htb]
\includegraphics[width=1\linewidth]{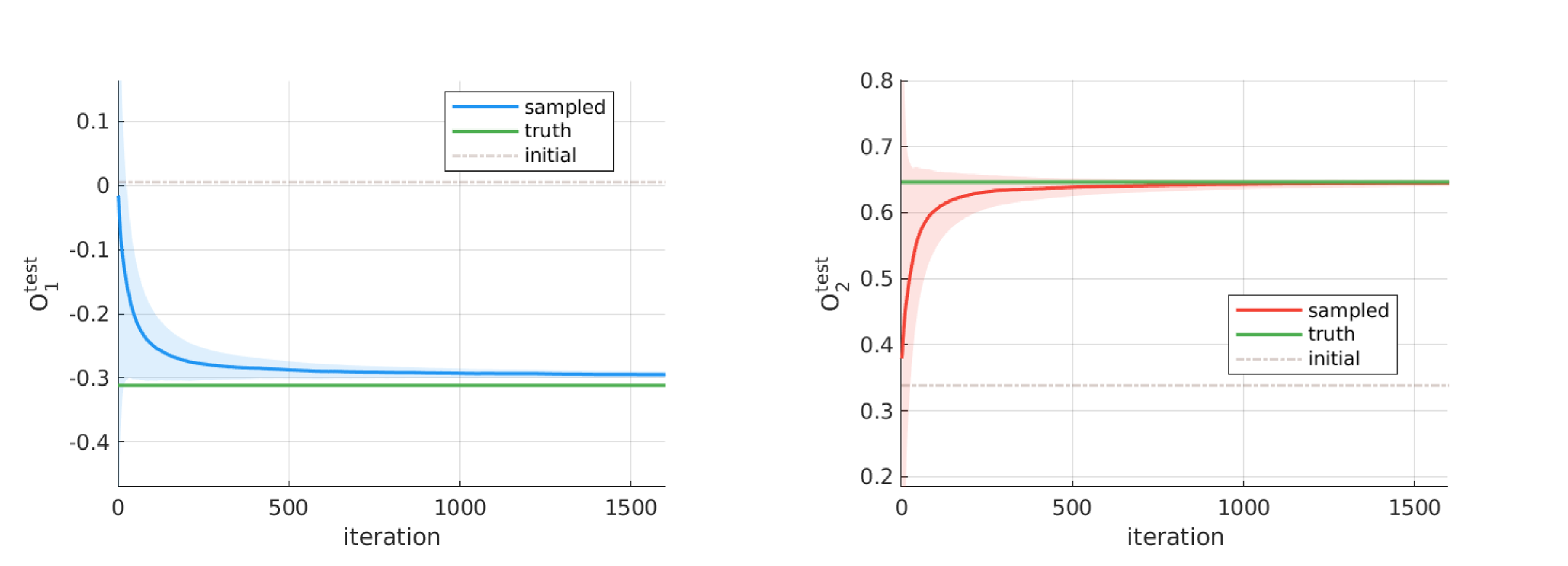}
\caption{
	Sampled test observations averaged along the proposal-decision iterations. In each of the two figures, the blue/red line corresponds to the averaged samples for	the 1st/2nd test observable. The shaded area indicates twice the standard error with respect to different independent runs. The green line stands for the ground-truth thermal average of the test observable. The faded brown line stands for the value sampled at the initial iteration by setting the potential $V$ to be the harmonic oscillation potential $V_o$.
}
\label{fig:1-level show-case averaged-observables}
\end{figure}

\begin{figure}[!htb]
\includegraphics[width=1\linewidth]{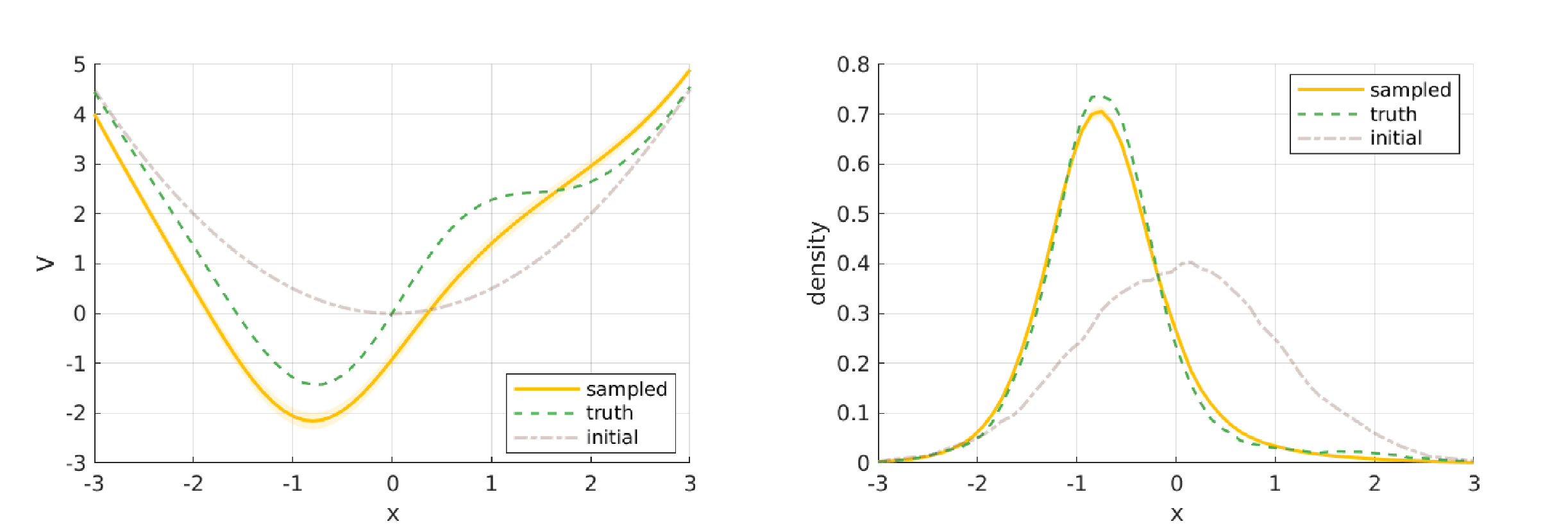}
\caption{
	Averaged sampled potential functions $\frac{1}{K_{total}} \sum_{k=1}^{K_{total}}V^{\left( k \right)}$ (left figure) and density functions (right figure). \\
	Shaded area indicates twice the standard error with respect to different independent runs. \\
	Green dashed line: obtained by setting the potential of the system to the ground-truth potential $V_{truth}$. \\
	Faded brown line: obtained from the harmonic oscillation potential $V_{o}$ (also used as the initial condition $V^{\left( 0 \right)}$).
}
\label{fig:1-level show-case averaged-potential, density}
\end{figure}

\subsection{Stability result: A numerical proof\label{sec:Stability-result A-numerical-proof}}

Recall the form of theorem on stability (Cor. \ref{cor:1-level d(E(A'))<d(y)}):
\[
\left|\mathbb{E}_{\pi^{\boldsymbol{y}_{1}}}\overline{O}\left[\mathfrak{q}\right]-\mathbb{E}_{\pi^{\boldsymbol{y}_{2}}}\overline{O}\left[\mathfrak{q}\right]\right|\lesssim_{r}\left|\left|\boldsymbol{y}_{1}-\boldsymbol{y}_{2}\right|\right|_{2},
\]
this inequality describes how the perturbation of output testing observations
are bounded by that of the input training observations. To be more specific,
we can interpret it in the following two ways:
\begin{enumerate}
\item Given noise covariance matrix $\Gamma_{\boldsymbol{\eta}}$, the theorem ensures
that the output predictions should rely on the input training observations
continuously, thus leading to a stable numerical algorithm.
\item More generally, the theorem also describes how consistent the output
predictions on test observables are if the noise covariance is specified.
Consider the following formal derivation. Let $\boldsymbol{y}^{*}$
be the thermal average of training observables and $V_{truth}$ be
the potential function. The distribution of the noisy training observation $\boldsymbol{y}$
is  $\mathcal{N}\left(\boldsymbol{y}^{*},\Gamma_{\boldsymbol{\eta}}\right)$,
merely a translation of the distribution of the noise $ \boldsymbol{\eta} $.
According to Jensen's inequality, the expectation of the input disturbance can be estimated by interchanging the expectation operator and the quadratic function:
\[
\mathbb{E}_{\boldsymbol{y}\sim\mathcal{N}\left(\boldsymbol{y}^{*},\Gamma_{\boldsymbol{\eta}}\right)}\left|\left|\boldsymbol{y}-\boldsymbol{y}^{*}\right|\right|_{2}\le\sqrt{\mathbb{E}_{\mathcal{N}\left(\boldsymbol{0},\Gamma_{\boldsymbol{\eta}}\right)}\left|\left|\boldsymbol{\boldsymbol{\eta}}\right|\right|_{2}^{2}}=\sqrt{\mathbf{Tr}\left(\Gamma_{\boldsymbol{\eta}}\right)}.
\]
On the other hand, assume that the negative log likelihood $\Phi\left(V;\boldsymbol{y}\right)$ admits a unique minimizer $V_{truth}$, then as $\mathbf{Tr}\left(\Gamma_{\boldsymbol{\eta}}\right)$
approaches zero, the measure $\mu^{\boldsymbol{y}^{*}}$ induced on
the potential space $W^{\boldsymbol{1}}$ converges to the delta
measure $\delta\left(V-V_{truth}\right)$, in which case the test prediction
$\mathbb{E}_{\mathfrak{q}\sim\pi^{\boldsymbol{y}}}\overline{O}\left[\mathfrak{q}\right] = \mathbb{E}_{V\sim\mu^{\boldsymbol{y}}} G^O\left(V\right) $
converges to $G^{O}\left(V_{truth}\right)$,
i.e. the ground-truth thermal average of test observables.
To sum up, if we draw a few training observations
from $\mathcal{N}\left(\boldsymbol{y}^{*},\Gamma_{\boldsymbol{\eta}}\right)$ and send them through the inversion process,
the average of the predictions on test observables can be written as a double expectation 
$
\mathbb{E}_{\boldsymbol{y}\sim\mathcal{N}\left(\boldsymbol{y}^{*},\Gamma_{\boldsymbol{\eta}}\right)}\mathbb{E}_{\mathfrak{q}\sim\pi^{\boldsymbol{y}}}\overline{O}\left[\mathfrak{q}\right]
$.
Thus the error between this average and the ground truth can be estimated by the following inequality
\begin{align}
\left|\text{error}\right|
&=\left|\mathbb{E}_{\boldsymbol{y}\sim\mathcal{N}\left(\boldsymbol{y}^{*},\Gamma_{\boldsymbol{\eta}}\right)}\mathbb{E}_{\mathfrak{q}\sim\pi^{\boldsymbol{y}}}\overline{O}\left[\mathfrak{q}\right]-G^{O}\left(V_{truth}\right)\right| \nonumber \\
&\le\left\lbrace\mathbb{E}_{\boldsymbol{y}\sim\mathcal{N}\left(\boldsymbol{y}^{*},\Gamma_{\boldsymbol{\eta}}\right)}\left|\mathbb{E}_{\mathfrak{q}\sim\pi^{\boldsymbol{y}}}\overline{O}\left[\mathfrak{q}\right]-\mathbb{E}_{\mathfrak{q}\sim\pi^{\boldsymbol{y}^{*}}}\overline{O}\left[\mathfrak{q}\right]\right|\right\rbrace
+ \left|\mathbb{E}_{\mathfrak{q}\sim\pi^{\boldsymbol{y}^{*}}}\overline{O}\left[\mathfrak{q}\right]-G^{O}\left(V_{truth}\right)\right| \nonumber \\
&\lesssim_{r} \left\lbrace\mathbb{E}_{\boldsymbol{y}\sim\mathcal{N}\left(\boldsymbol{y}^{*},\Gamma_{\boldsymbol{\eta}}\right)}\left|\left|\boldsymbol{y}-\boldsymbol{y}^{*}\right|\right|_{2}\right\rbrace 
+ \left|\mathbb{E}_{V\sim\mu^{\boldsymbol{y}^{*}}}G^{O}\left(V\right)-G^{O}\left(V_{truth}\right)\right| \nonumber \\
&\lesssim_{r}\sqrt{\mathbf{Tr}\left(\Gamma_{\boldsymbol{\eta}}\right)} + \mathbf{Tr}\left(\Gamma_{\boldsymbol{\eta}}\right) \cdot \sup_{V\in \boldsymbol{W}^1} \delta^2 G^O\left(V\right).
\label{eq:1-level stability-result error-formal-estimation}
\end{align}

When the covariance matrix is small enough, we can ignore the second term since
it is of higher order than the first one. However, we must point out these
derivations are formal since we assume that the negative log likelihood
$\Phi\left(V;\boldsymbol{y}\right)$ has only one minimizer $V_{truth}$ and 
the second order variational derivative $\delta^2G^O\left(V\right)$ is bounded
on $\boldsymbol{W}^1$.
The conditions are not easy to prove since the first
assumption is equivalent to finding a unique $V$ that can reproduce the
training observables, which is discussed in \cite{Mehats2010} with the density
function known.
% Generally speaking, a unique minimizer may exist when
% the number of training observables $N_1$ is larger than the dimension of the truncated potential space $\boldsymbol{W}^1_L$.
The error analysis when multiple minimizers exist is possible, but it is far beyond the scope of the current paper. As is also shown in this section, the algorithm does not rely on the unique minimizer assumption to perform well.
\end{enumerate}
In the following experiments, we will assume different scales for
the noise covariance and compare the corresponding test observations.
As in the previous section, the noise covariance is set to $\Gamma_{\boldsymbol{\eta}}=\widetilde{\Gamma_{\eta}}I^{N_{1}\times N_{1}}$,
where $\widetilde{\Gamma_{\eta}}=0.01,0.03,0.1,0.3$.

\begin{figure}[!htb]
\includegraphics[width=1\linewidth]{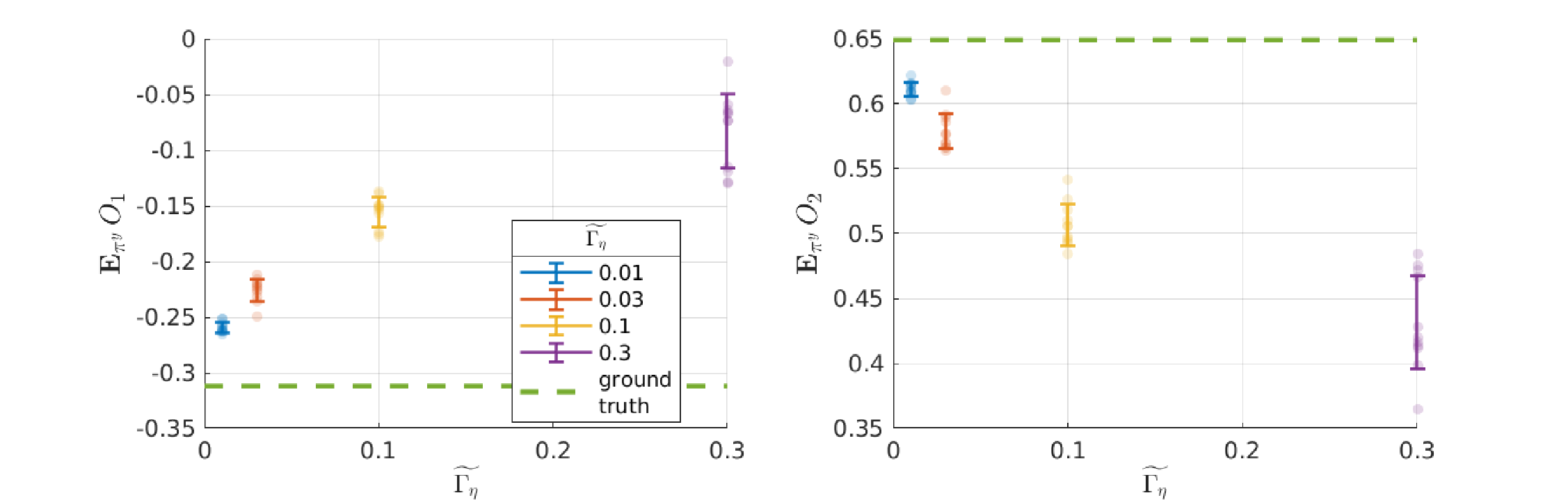}
\caption{
	Testing observations compared among experiments if different scales of noise covariance $\widetilde{\Gamma_{\eta}}$ are assumed. Horizontal axis: the scale of noise variance $\widetilde{\Gamma_{\eta}}$. Vertical axis: test observations, where the green dashed line stands for the ground truth. The solid bars stand for the mean and standard error of test observations (faded dots) obtained in numerical experiments.
}
\label{fig:1-level stability-result test-obs-gamma}
\end{figure}

Judging from Fig. \ref{fig:1-level stability-result test-obs-gamma}
we can see that a smaller $\widetilde{\Gamma_{\eta}}$ will lead to a more accurate mean and a smaller variance of the sampled test observations. Furthermore, if we fit
the error with the scale of noise covariance $\widetilde{\Gamma_{\eta}}=\frac{1}{N_{T}}\mathbf{Tr}\left(\Gamma_{\boldsymbol{\eta}}\right) $ ($N_T$ is the fixed number of training observables),
the slope in log scale (see
Fig. \ref{fig:1-level stability-result test-err-gamma}) is around $\frac{1}{2}$, which is consistent
with Eqn. \ref{eq:1-level stability-result error-formal-estimation}.

\begin{figure}[!htb]
\begin{centering}
\includegraphics[width=0.6\linewidth]{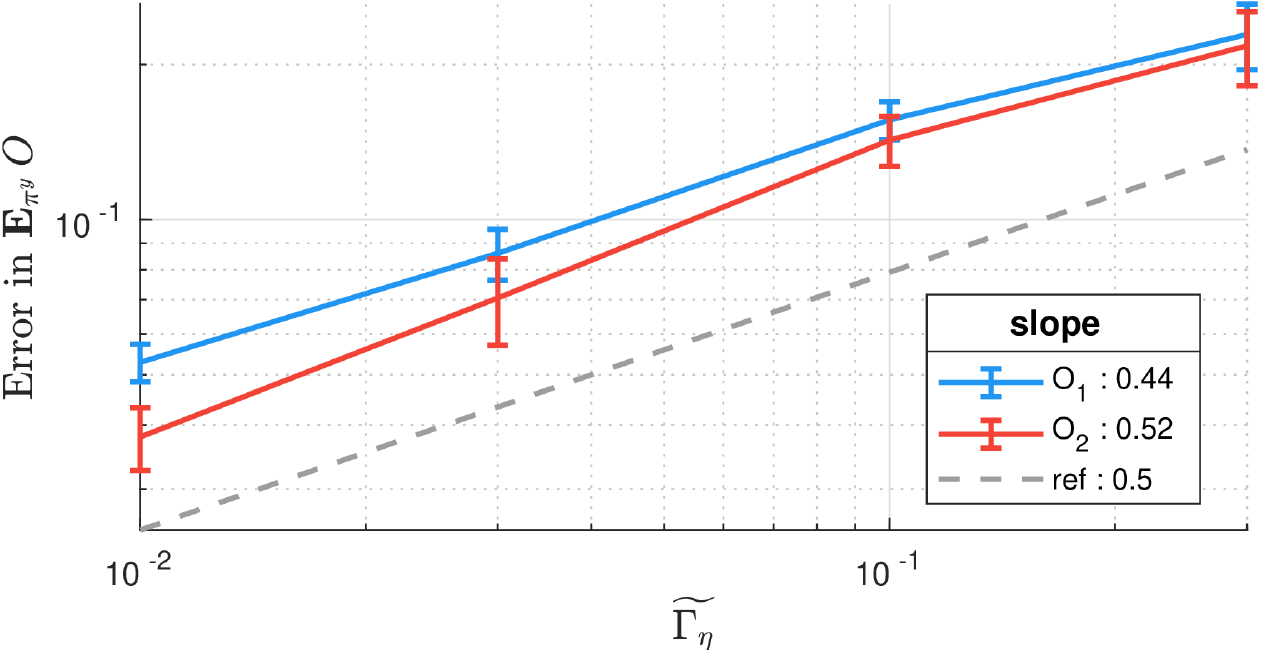}
\par\end{centering}
\caption{
	Log scale figure for test observation errors and noise covariance scales $\widetilde{\Gamma_{\eta}}$, indicating that assuming smaller noise covariances can lead to more consistent test observations sampled.
	Color folded line: averaged test observation errors obtained from numerical experiments. Grey dashed line: reference line with slope $\frac{1}{2}$. The fitted slopes for two folded lines are listed in the legend.
}
\label{fig:1-level stability-result test-err-gamma}
\end{figure}

\subsection{Numerical study on 2-Level system\label{sec:Numerical-Study-on-2-Level-System}}

\subsubsection{Problem formulation}

In the previous sections, we have build a rigorous theory on Bayesian
inversion in the quantum thermal average problem. It is tempting to
extend this theory into a broader area where the quantum system is associated 
with multiple electronic states and
 non-adiabatic effects are taken into account. Under a few technical
assumptions, the extension can be quite straightforward since the
essential procedures are kept the same: the potential function can
induce a probability measure on the bead configuration space, while
the thermal average can be compared to the ground truth of training
observations to induce a probability measure on the potential function
space. The prime difficulty lies in the forward problem of finding
an efficient way to sample the distribution on the bead configuration
space, which is concerned in \cite{Lu} and PIMD-SH is thus developed.
We refer to Sec. \ref{sec:RPR-in-Two-level-Systems} in the Appendix
for details. In the following sections we will basically transfer
the algorithm proposed for 1-level systems to 2-level systems and demonstrate
corresponding numerical studies.

\subsubsection{Notations}
The space of 2-level potential functions $\boldsymbol{V}$ is defined
as
\[
W^{\boldsymbol{2}}\triangleq\left\{ \boldsymbol{V}=\left(\begin{matrix}V_{00} & V_{01}\\
V_{01} & V_{11}
\end{matrix}\right):V_{00},V_{11}\in W^{\boldsymbol{1}},V_{01}\in W^{g}\right\} 
\]
and $W^{g}$ denotes the following mixed-Gaussian-component space
\[
W^{g}\triangleq\left\{ V_{01}:V_{01}\left(x\right)=\sum_{i=1}^{N_{\text{off}}}A_{i}\exp\left(-\frac{\left(x-c_{i}\right)^{2}}{2\sigma_{i}^{2}}\right)\right\} 
\]
where we have the following assumptions:
\begin{itemize}
	\item the off-diagonal potential component number $N_{\text{off}}$, location series $\left\{ c_{i}\right\} $ and derivation series $\left\{ \sigma_{i}\right\} $ are known and fixed,
	\item each amplitude component $A_{i}\in\mathbb{R}^{+}$ is unknown.
\end{itemize}

\subsubsection{Algorithm explanation}

There are two main differences compared from the 1-level algorithm
(Alg. \ref{alg:1-level algorithm}):
\begin{enumerate}
\item The proposal of $\boldsymbol{V}$: since we are now dealing with $\boldsymbol{V}\in W^{\boldsymbol{2}}$,
not only the diagonal terms $V_{00}$ and $V_{11}$ are sampled as
in the 1-level case, but also the off-diagonal term should be sampled.
The prior for $V_{01}$ (essentially for $\left\{ A_{i}\right\} $
since the other parameters are fixed) is chosen as the exponential
distribution, which can be efficiently proposed by Alg. \ref{alg:local-proposal-kernel exponential-distribution}.
\item The sampling of $\left(\boldsymbol{q},\boldsymbol{l}\right)$: it
is accomplished by the implementation of PIMD-SH. We refer to Sec.
\ref{sec:RPR-in-Two-level-Systems} in the Appendix for details.
\end{enumerate}

\subsubsection{Numerical study set-up}

Here we conduct tests for the following system:
\begin{align*}
V_{00} & =\frac{1}{2}x^{2}-\frac{3}{2}\phi_{1}\left(x\right)\\
% Green N500
V_{01} & =\exp\left(-2x^{2}\right)\\
% Purple N500
V_{11} & =\frac{1}{2}x^{2}-\frac{3}{4}\phi_{0}\left(x\right)-\frac{3}{2}\phi_{1}\left(x\right)-\phi_{2}\left(x\right)
% Green N700
\end{align*}
where the off-diagonal potential has $N_{\text{off}}=1$ Gaussian component and $A_{1}=1$ (unknown and to be recovered), $\sigma_{1}=0.5,c_{1}=0$ (fixed and known beforehand). The two diagonal-potentials
almost intersect near $x=0$, where $V_{01}$ is significantly positive, providing chances for hopping between
layers. Besides, the particle mass $M$ is set to $10$, truncation
level $L$ to $4$ and bead number $N$ to $8$.

The observables we used in training/testing have either diagonal
or off-diagonal non-zero entries only.
For those only with diagonal entries, we simply set the two diagonal entry to be the same, while those only with off-diagonal entries have the same off-diagonal real-valued entries since the observables are Hermite.
The training observables have diagonal Gaussian entries in between
$\left[-2,2\right]$ and off-diagonal Gaussian entries in between $\left[-1,1\right]$.
The testing observables are
\[
\begin{array}{ccc}
O_1 = \left(\begin{array}{cc}
o_1 & \\ & o_1
\end{array}\right)&, o_1 = \exp\left(-\frac{\left(x-1.25\right)^{2}}{4}\right) \\ % Teal N500
O_2 = \left(\begin{array}{cc}
o_2 & \\ & o_2
\end{array}\right)&, o_2 = \exp\left(-\frac{\left(x+0.25\right)^{2}}{4}\right) \\ % Light Blue N500
O_3 = \left(\begin{array}{cc}
& o_3 \\ o_3
\end{array}\right)&, o_3 = \exp\left(-8\left(x-0.1\right)^{2}\right) \\ % Red N500
O_4 = \left(\begin{array}{cc}
& o_4 \\ o_4
\end{array}\right)&, o_4 = \exp\left(-8\left(x-0.3\right)^{2}\right)     % Orange N500
\end{array}
\]
as shown in Fig \ref{fig:2-level potential-observable-illustration}.

\begin{figure}[!htb]
\includegraphics[width=1\linewidth]{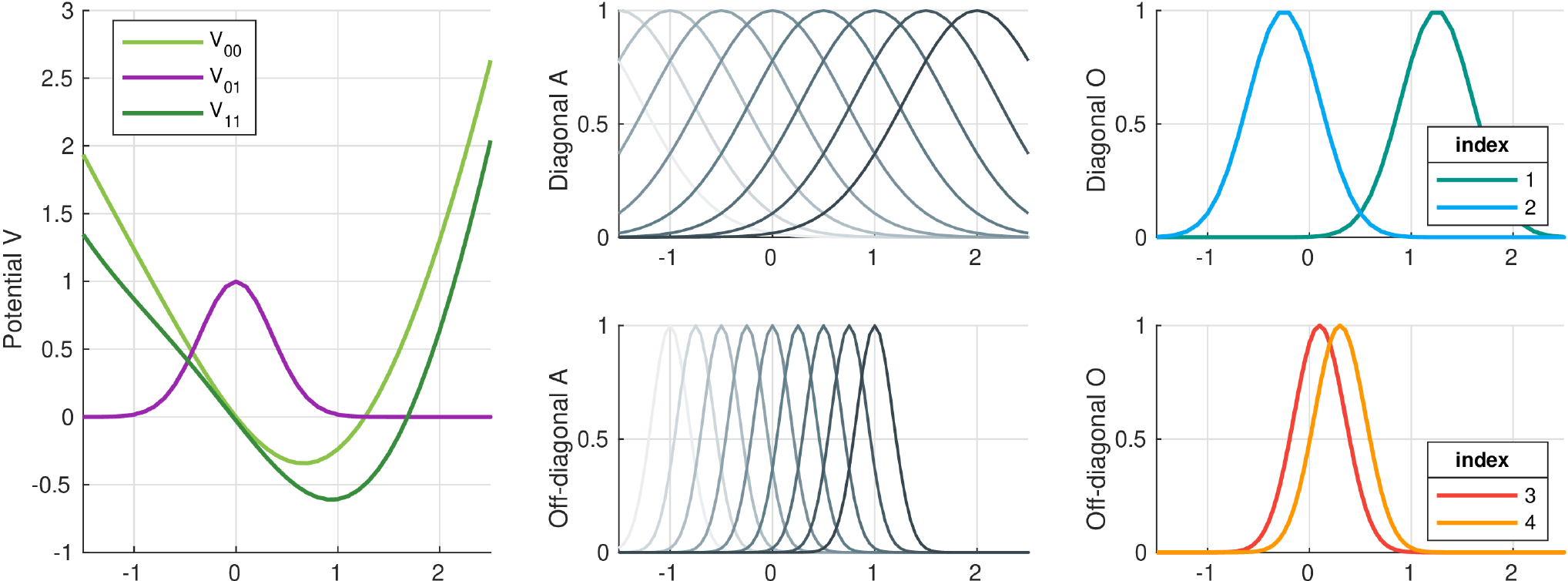}
\caption{Illustration on ground-truth potential (left-most) and training (middle)/testing (right-most)
observables in 2-level system. Top rows show observables with only diagonal entries and bottoms rows show observables with only off-diagonal entries.}

\label{fig:2-level potential-observable-illustration}
\end{figure}

\subsubsection{Numerical results and discussions}
Likewise, we demonstrate that the algorithm in the 2-level setting can recover potential landscape and predict test observables. As shown in  Fig. \ref{fig:2-level test-observation-sampled} and \ref{fig:2-level potential-sampled}, the sampled average will start from the initial value (marked by the brown line) and converge to the proposed average $\mathbb{E}_{\pi^{\boldsymbol{y}}}O\left(\mathfrak{q}\right)$, which is quite close to the ground truth (marked by the green line). The potential learned exhibits a tendency to be deeper near $x=1$, which captures the landscape of ground truth potentials.

\begin{figure}[!htb]
\begin{centering}
\includegraphics[width=0.75\linewidth]{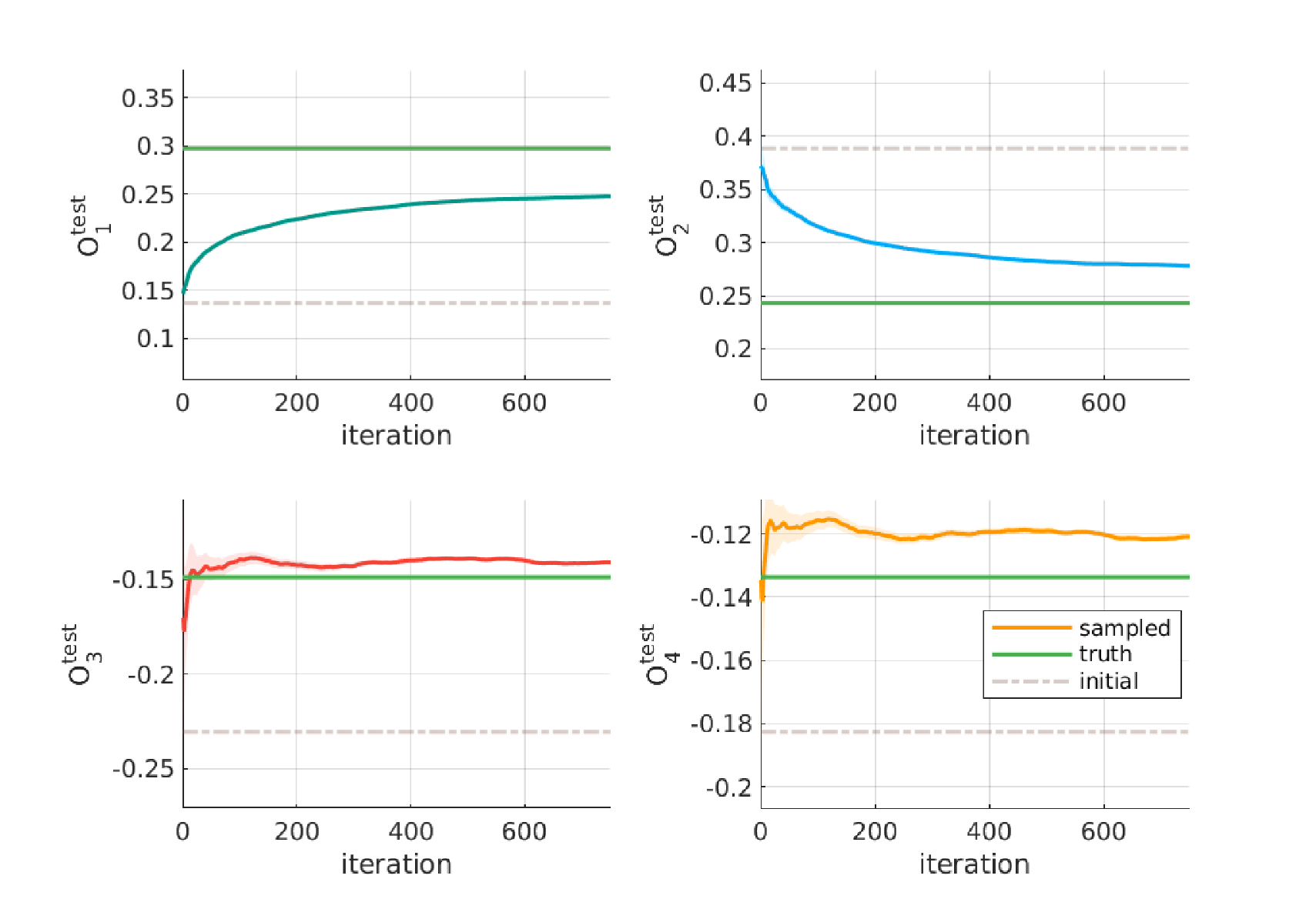}
\par\end{centering}
\caption{Observations for test observables averaged along the proposal-decision iterations. Top two figures: diagonal-only testing observables $\widehat{O_1}$ (teal), $\widehat{O_2}$ (light blue). Bottom two figures: off-diagonal-only testing observables $\widehat{O_3}$ (red), $\widehat{O_4}$ (orange). Green solid line: ground truth. Faded brown dashed line: a guess based on the initial potential function.}

\label{fig:2-level test-observation-sampled}
\end{figure}

\begin{figure}[!htb]
\includegraphics[width=1\linewidth]{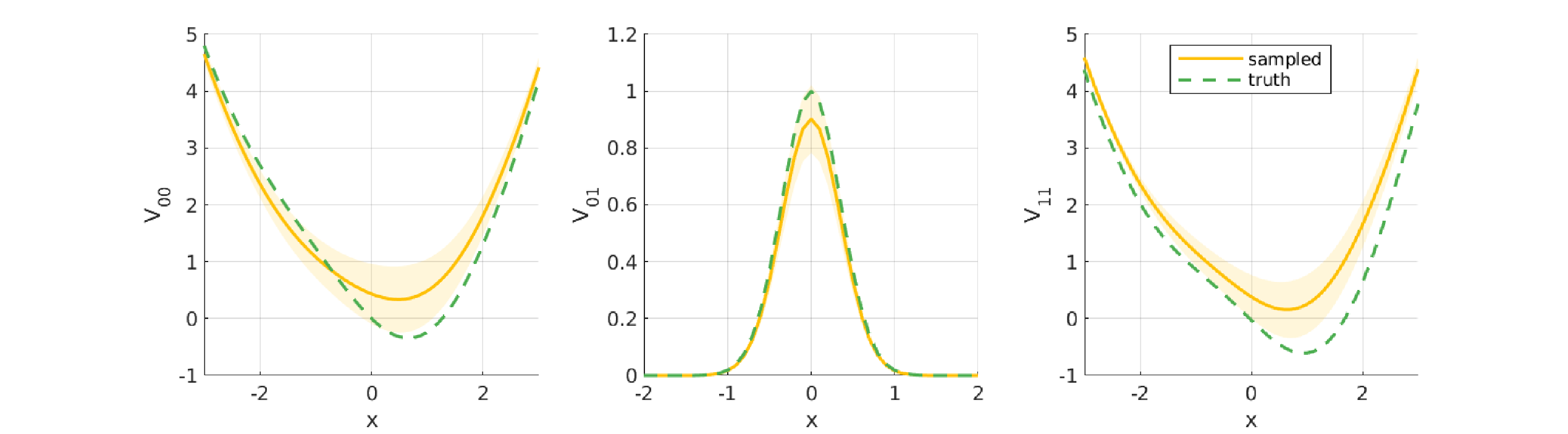}
\caption{Averaged sampled potentials and densities. The left and right figure
demonstrate the diagonal potential landscape on 2 levels, while the
middle one demonstrates the off-diagonal potential $V_{01}$. Blue
area indicates twice the standard error. Yellow dashed line indicates
the ground truth obtained by setting the potential $V$ to be $V_{truth}$.}

\label{fig:2-level potential-sampled}
\end{figure}

In a non-adiabatic system, it is quite reasonable to speculate that
training with off-diagonal-only observables (i.e. vanish along the diagonal line and only having off-diagonal terms) will help recover the potential landscape more accurately. Fig.
\ref{fig:2-level with-without-offdiagonal} shows the comparison between
results obtained by training the system with and without off-diagonal-only
observables. It is especially clear that the `with' condition outperforms
in the off-diagonal-only test observables in terms of convergence speed and bias.

\begin{figure}[!htb]
\begin{centering}
\includegraphics[width=0.8\linewidth]{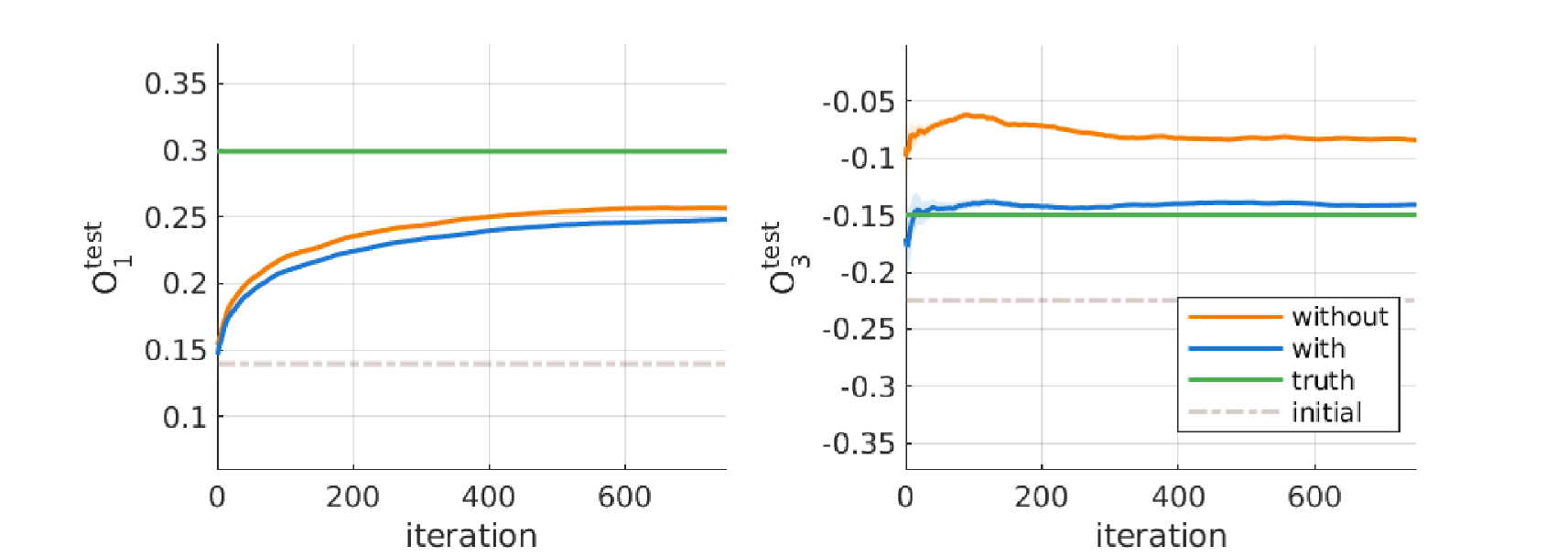}
\par\end{centering}

\caption{A comparison between training with (blue line) and without (deep orange line) off-diagonal-only observables. The ground truth is marked as the green solid line, and the initial result obtained at the beginning of iterations is marked as the faded brown dashed line. Results on the 2nd and 4th test observables are similar and are not shown in the figure.}

\label{fig:2-level with-without-offdiagonal}
\end{figure}

\section{Conclusion\label{sec:Conclusion}}

In this paper, we proposed a novel method regarding the inversion
problem of quantum thermal average. By adopting the Bayesian inversion
framework, we not only obtain a posterior distribution of potential
functions based on finite noisy observations, but also the stability
result is established under the framework and the numerical algorithm
is proved to be efficient. In the numerical section, we have shown that
the ensemble average of bounded testing observables can be accurately
sampled, and furthermore the error can be controlled by the assumed
perturbation in the input training data, which testifies our theoretical
result (Thm. \ref{thm:1-level d(=00005Cpi^y)<d(y) (main result)}). We also demonstrated how to apply the algorithm to the non-adiabatic regime. Since the tests in the 2-level setting are preliminary and the accuracy is not satisfactory, it indeed requires more attention and effort on the improvement.
One way is to adopt a more efficient sampling algorithm, e.g. the Delayed Rejection Adaptive Metropolis \cite{Haario2006} (with both techniques used together or separately). Another way is to dig deep into the intrinsic mathematical structure of this problem and further improve the precision.
In addition, it is not clear yet what observables are favorable in order to have a consistent approximation of the potential function.

Here we ought to mention that this work can easily extend to related problems.
First, our method can be extended to systems in higher dimensions. One can easily verify that Thm. \ref{thm:1-level d(=00005Cpi^y)<d(y) (main result)} still holds if the loop space $X$ is extended to $\mathcal{L}\mathbb{R}^{d}$ and potential space $W^{\boldsymbol{1}}$ is modified accordingly. As for the numerical implementation, we can expand the potential onto
tensor product of Hermite basis.
Second, judging from the numerical Sec.\ref{sec:Numerical-Study-on-2-Level-System}, our method fits multi-level systems quite well. A similar theoretical stability result can also be obtained if the Boltzmann distribution follows the form in the single-level system. Future investigations into a more general function space $W^g$ setting for the off-diagonal potential functions may yield better results.
At the same time, the framework proposed in this work is compatible with all kinds of solvers for the forward problem, so the methods developed in \cite{Liu2018,Lu2018a,Tao2018} can be seamlessly integrated into the algorithm. Besides, if given multiple training observations, we can also make use of the Ensemble Kalman Filter (EnKF) \cite{Mandel2009} to improve the accuracy.

\bibliographystyle{apalike}
\bibliography{ref_2}

\section*{ACKNOWLEDGMENTS}
Z. Zhou was supported by NSFC grant No. 11801016. Z. Zhou would like to thank Jianfeng Lu and Thomas F. Miller III for helpful discussions. Z. Chen would like to thank Tiejun Li for guidance in Chen's undergraduate study.

\appendix

\section{Analysis Supplements\label{sec:Analysis-Supplyments}}
\begin{lem}
	\label{lem:W1 complete}
	The vector space $W^{\boldsymbol{1}}\triangleq L^2\left(\mathbb{R}\right) \bigcap L^\infty\left(\mathbb{R}\right)$ is a Banach space under the norm
	\[
	\left|\left|f\right|\right|_{W^{\boldsymbol{1}}}\triangleq\left|\left|f\right|\right|_{L^2\left(\mathbb{R}\right)}+\left|\left|f\right|\right|_{L^{\infty}\left(\mathbb{R}\right)}.
	\]
\end{lem}
\begin{proof}
	It suffices to show that a Cauchy sequence $\left\lbrace f_n \right\rbrace$ in $W$ converges in $W$. In fact, $\left\lbrace f_n \right\rbrace$ is a Cauchy sequence both in $L^{2}\left(\mathbb{R}\right)$ and in $L^{\infty}\left(\mathbb{R}\right)$. Denote the corresponding limit as $f_n \xrightarrow{L^2\left(\mathbb{R}\right)} h$ and $f_n \xrightarrow{L^\infty\left(\mathbb{R}\right)} \widetilde{h}$.
	
	Restricted on a given interval $\left[a,b\right] $, notice that $\left|\left|\cdot\right|\right|_{L^2\left( \left[ a,b \right]\right)} \le\left|\left|\cdot\right|\right|_{L^2\left(\mathbb{R}\right)}$ and $\left|\left|\cdot\right|\right|_{L^\infty\left( \left[ a,b \right]\right)} \le\left|\left|\cdot\right|\right|_{L^\infty\left(\mathbb{R}\right)}$, we have $h\in L^2\left( \left[ a,b \right] \right) $ and $\widetilde{h}\in L^\infty\left( \left[ a,b \right] \right) $.
	Notice that for $g_n\triangleq f_n-\widetilde{h}\in L^\infty\left( \left[ a,b \right]\right)$,
	\[
		\left|\left|g_n\right|\right|_{L^2\left( \left[ a,b \right]\right)}
		\le \sqrt{\left(b-a\right)\left|\left|g_n\right|\right|_{L^\infty\left( \left[ a,b \right]\right)}^2}
		\le \sqrt{b-a} \left|\left|g_n\right|\right|_{L^\infty\left( \left[ a,b \right]\right)},
	\]
	so $g_n$ converges to $0$ in $L^2\left( \left[ a,b \right] \right) $, i.e. $f_n \xrightarrow{L^2\left( \left[ a,b \right] \right)} \widetilde{h}$.
	But $f_n$ also converges to $h$ in $L^2\left( \left[ a,b \right] \right) $, so $h$ and $\widetilde{h}$ a.e. agree with each other on $\left[ a,b \right]$.
	
	Then, since $h$ and $\widehat{h}$ a.e. agree on $\left[ -N,N \right]$, they a.e. agree on $\mathbb{R}$, so $f_n$ converges to $h\in L^2\left(\mathbb{R}\right) \bigcap L^\infty\left(\mathbb{R}\right)$ under the $\left|\left|\cdot\right|\right|_{W^{\boldsymbol{1}}}$ norm.
\end{proof}

\begin{defn}
	\label{def:hermite-function}
The Hermite function sequence $\left\{ \phi_{n}\right\} $ is defined
as 
\[
\phi_{n}\left(x\right)\triangleq\left(-1\right)^{n}\cdot\left(2^{n}n!\sqrt{\pi}\right)^{-\frac{1}{2}}\cdot e^{\frac{x^{2}}{2}}\cdot\dfrac{d^{n}}{dx^{n}}e^{-x^{2}}.
\]
\end{defn}

\begin{thm}
\label{thm:corollary-7.22-continuity-extension}
Let $ D\subset \mathbb{R}^d $ be a compact domain, $\left\{ \gamma_{k}\right\} _{k\ge0}$
be a real-valued sequence and $\left\{ \xi_{k}\right\} _{k\ge0}$
be countably many centered i.i.d. random variables with bounded moments
of all orders. Assume that $\left\{ \psi_{k}\right\} _{k\ge0}$ is
a sequence of H\"{o}lder functions satisfying
\[
\left|\psi_{k}\left(x\right)\right|\le C_{1},\ \forall k\ge0,x\in D
\]
\[
\left|\psi_{k}\left(x\right)-\psi_{k}\left(y\right)\right|\le C_{2}k^{t}\left|x-y\right|^{\alpha},\ \forall k\ge0,x\in D
\]
where $C_{1},C_{2}$ are constants and $\alpha\in(0,1],t>0$. Suppose
there is some $\delta\in\left(0,2\right)$ such that 
\[
S_{1}\triangleq\sum_{k\ge0}\gamma_{k}^{2}<\infty\ \text{and}\ S_{2}\triangleq\sum_{k\ge0}\gamma_{k}^{2-\delta}k^{t\delta}<\infty.
\]
Then $u$ defined by 
\[
u\triangleq\sum_{k\ge0}\gamma_{k}\xi_{k}\psi_{k}
\]
is a.s. finite for every $x\in D$, and $ u\in C^{0,\alpha'}\left( D\right)$ for every $\alpha'<\alpha\delta/2$
(i.e. $u$ is H\"{o}lder continuous for every exponent smaller than $\alpha\delta/2$).
\end{thm}

This theorem is a direct corollary which can be found in \cite{Dashti2017},
Page 91.

\begin{prop}
	\label{prop:function-series-converge-in-L-infty}
	Assume the positive sequence $\gamma_{j}=\mathcal{O}\left(j^{-s}\right)$ with $s>1$
	and $\left\lbrace \xi_{j} \right\rbrace_{j=1}^\infty$ are i.i.d. random variables with finite second moment.
	Denote the Hermite functions as $\left\{ \phi_{n}\right\} $ defined in \ref{def:hermite-function}.
	Then the function series
	$\sum_{j=0}^{\infty}\xi_{j}\gamma_{j}\phi_{j}$ $\mathbb{P}$-a.s. converges in
	$W^{\boldsymbol{1}} = L^{\infty} \cap L^2$ with norm $\left|\left|\cdot\right|\right|_{W^{\boldsymbol{1}}} = \left|\left|\cdot\right|\right|_{L^2} + \left|\left|\cdot\right|\right|_{L^\infty}$.
\end{prop}

\begin{proof}
	We will use the following lemma without proof: let $\left\{ I_{j}\right\} _{j=0}^{\infty}$
	be a sequence of $\mathbb{R}^{+}$-valued independent random variables,
	then the following equivalence holds:
	\[
	\sum_{j=0}^{\infty}I_{j}<+\infty\ \mathbb{P}\text{-a.s.} \Longleftrightarrow\sum_{j=0}^{\infty}\mathbb{E}\left(I_{j}\land1\right)<\infty.
	\]
	
	For the $L^\infty$ part, we denote $I_{j}^{\left(1\right)}=\left|\left|\xi_{j}\gamma_{j}\phi_{j}\right|\right|_{L^{\infty}}$. 
	By Cram\'{e}r's inequality $\left|\left|\phi_{n}\right|\right|_{L^{\infty}}\le\pi^{-\frac{1}{4}}$,
	we have
	\[
	\sum_{j=0}^{\infty}\mathbb{E}I_{j}^{\left(1\right)}
	= \sum_{j=0}^{\infty}\gamma_{j}\left|\left|\phi_{j}\right|\right|_{L^{\infty}}\mathbb{E}\left|\xi_{j}\right|
	\le \pi^{-\frac{1}{4}}\mathbb{E}\left|\xi_{j}\right| \, \sum_{j=0}^{\infty}\gamma_{j}.
	\]
	
	For the $L^2$ part, we denote $I_{j}^{\left(2\right)}=\left|\left|\xi_j \gamma_j\phi_j\right|\right|_{L^2}$, then
	\[
	\sum_{j=0}^{\infty}\mathbb{E}I_{j}^{\left(2\right)}
	= \mathbb{E}\left|\xi_{j}\right|\,\sum_{j=0}^{\infty}\gamma_{j}
	\]
	Since we already assumed that $\gamma_{j}=\mathcal{O}\left(j^{-s}\right)$
	with $s>1$, the infinite sums $\sum_{j=0}^\infty \gamma_j$ and $\sum_{j=0}^\infty \gamma_j^2$ are finite, so
	\begin{align}
	\sum_{j=0}^{\infty}\mathbb{E}\left(\left|\left|\xi_{j}\gamma_{j}\phi_{j}\right|\right|_{W^{\boldsymbol{1}}}\land1\right)
	& \le \sum_{j=0}^{\infty}\mathbb{E}I_{j}^{\left(1\right)} + \sum_{j=0}^{\infty}\mathbb{E}I_{j}^{\left(2\right)} \\
	& \le \left(1+\pi^{-\frac{1}{4}}\right)\mathbb{E}\left|\xi_{j}\right|\,\sum_{j=0}^{\infty}\gamma_{j} \\
	&< \infty.
	\end{align}
	According to the lemma we stated at the beginning of the proof, the infinite sum 
	$\sum_{i=0}^{\infty}\left|\left|\xi_{j}\gamma_{j}\phi_{j}\right|\right|_{W^{\boldsymbol{1}}}$ $\mathbb{P}$-a.s. converges. By the completeness of $W^{\boldsymbol{1}}$, the function series $\sum_{i=0}^{\infty}\xi_{i}\gamma_{i}\phi_{i}$ $\mathbb{P}$-a.s. converges to a $W^{\boldsymbol{1}}$ function.
\end{proof}

\begin{prop}
	\label{prop:holder-continuity-of-function-series}
	Assume the positive sequence $\gamma_{j}=\mathcal{O}\left(j^{-s}\right)$ with $s>\frac{1}{2}$ and $\xi_{j}\stackrel{i.i.d.}{\sim}\mathcal{N}\left(0,1\right)$.
	Denote the Hermite functions defined in \ref{def:hermite-function} as $\left\{ \phi_{n}\right\} $.
	Then for all $\beta<\frac{4s-2}{4s+1}$, the function series $\sum_{j=0}^{\infty}\xi_{j}\gamma_{j}\phi_{j}$
	$\mathbb{P}$-a.s. converges in $C^{0,\beta}\left(\mathbb{R}\right)$.
\end{prop}

\begin{proof}
	First we fix the domain $D_R$ restricted to $\left[-R,R\right]$.
	By Thm. \ref{thm:corollary-7.22-continuity-extension}, the series
	of interest has proper regularity as long as $S_{1}$ and $S_{2}$ are
	finite. The key of this proof is to identify constants $t$ and $\alpha$
	satisfying
	\[
	\left|\phi_{n}\left(x\right)-\phi_{n}\left(y\right)\right|\le C_{2}n^{t}\left|x-y\right|^{\alpha}
	\]
	for any $x,y\in D_R$ and $n\in\mathbb{N}$.
	
	For Hermite basis functions, $\arg\max\left|\phi_{n}'\right|=0$ if
	$n$ is odd. If we take $p=\frac{n+1}{2}$, we can have the following
	equation after some tedious calculations:
	\[
	\left|\phi_{n}'\left(0\right)\right|=\frac{\sqrt{\left(2p-1\right)!}}{2^{p-1}\left(2\pi\right)^{\frac{1}{4}}\left(p-1\right)!}.
	\]
	By Stirling formula, 
	\[
	\log\left|\phi_{n}'\left(0\right)\right|=\frac{1}{4}\log\left(n+1\right)-\frac{1}{2}\log\pi-\frac{1}{8\left(n+1\right)}+\mathcal{O}\left(n^{-2}\right),
	\]
	which indicates that $t=\frac{1}{4}$.
	
	By assumption, $\gamma_{j}=\mathcal{O}\left(j^{-s}\right)$, we have
	\[
	S_{1}<\infty\Longleftrightarrow s>\frac{1}{2},
	\]
	\[
	S_{2}<\infty\Longleftrightarrow-\left(2-\delta\right)s+t\delta<-1\Longleftrightarrow\delta<\frac{2s-1}{s+t}=\frac{8s-4}{4s+1}.
	\]
	So by Thm. \ref{thm:corollary-7.22-continuity-extension}, the function series $\sum_{j=0}^{\infty}\xi_{j}\gamma_{j}\phi_{j}$ is H\"{o}lder continuous on $D_R$ for every exponent smaller than $\frac{4s-2}{4s+1}$.
	Since $\mathbb{R}=\bigcup_{R=1}^\infty D_R$ and the H\"{o}lder continuity of the function series $\mathbb{P}$-a.s. holds for countably many $R$, it also $\mathbb{P}$-a.s. holds on $\mathbb{R}$.
\end{proof}

\section{RPR in Two-level Systems\label{sec:RPR-in-Two-level-Systems}}

\subsection{Ring polymer representation}

Without loss of generality, a few assumptions have been made in \cite{Lu}:
\begin{enumerate}
\item The total level of energy in the system is 2.
\item The potential function $V=\left(\begin{array}{cc}
V_{00} & V_{01}\\
V_{01} & V_{11}
\end{array}\right)$ is Hermite and the off-diagonal term $V_{01}$ keeps its sign.
\item The observable $\widehat{A}$ only depends on position $q$.
\end{enumerate}
To fully describe the bead, a level index vector $\boldsymbol{l}\in\left\{ 0,1\right\} ^{N}$
is needed besides the position configuration $\boldsymbol{q}\in\mathbb{R}^{N}$.
By a similar induction and introduction of auxiliary momentum variable
$\boldsymbol{p}$ as in the single level system, we arrive at the
following formula for the thermal average
\begin{equation}
\langle\widehat{A}\rangle =\frac{1}{\mathcal{Z}_{N}}\int_{\mathbb{R}^{N}}\int_{\mathbb{R}^{N}}\sum_{\boldsymbol{l}\in\left\{ 0,1\right\} ^{N}}W_{N}\left[A\right]e^{-\beta_{N}H_{N}\left(\boldsymbol{q},\boldsymbol{p},\boldsymbol{l}\right)}d\boldsymbol{q}d\boldsymbol{p}+\mathcal{O}\left(N^{-2}\right),\label{eq:2-level-thermal-average}
\end{equation}
where the weight function is given by 
\begin{equation}
W_{N}\left[A\right]\left(\boldsymbol{q},\boldsymbol{p},\boldsymbol{l}\right)\triangleq\frac{1}{N}\sum_{k=1}^{N}\left\{ \left\langle l_{k}|A\left(q_{k}\right)|l_{k}\right\rangle -\boldsymbol{sgn}\left(V_{l_{k}\overline{l_{k}}}\right)\left\langle l_{k}|A\left(q_{k}\right)|\overline{l_{k}}\right\rangle \exp\left[\beta_{N}\left(\left\langle l_{k}|G_{k}|l_{k+1}\right\rangle -\left\langle \overline{l_{k}}|G_{k}|l_{k+1}\right\rangle \right)\right]\right\} ,\label{eq:2-level-thermal-average-weight}
\end{equation}
the Hamiltonian given by
\begin{equation}
H_{N}\left(\boldsymbol{q},\boldsymbol{p},\boldsymbol{l}\right)=\sum_{k=1}^{N}\left\langle l_{k}|G_{k}|l_{k+1}\right\rangle ,\label{eq:2-level-thermal-average-hamiltonian}
\end{equation}
the matrix $G_{k}$ given by 
\begin{equation}
\left\langle l|G_{k}|l'\right\rangle =\begin{cases}
\frac{p_{k}^{2}}{2M}+\frac{M\left(q_{k}-q_{k+1}\right)^{2}}{2\beta_{N}^{2}}+V_{ll}\left(q_{k}\right)-\frac{1}{\beta_{N}}\log\left(\cosh\left(\beta_{N}\left|V_{01}\left(q_{k}\right)\right|\right)\right) & l=l'\\
\frac{p_{k}^{2}}{2M}+\frac{M\left(q_{k}-q_{k+1}\right)^{2}}{2\beta_{N}^{2}}+\frac{V_{00}\left(q_{k}\right)+V_{11}\left(q_{k}\right)}{2}-\frac{1}{\beta_{N}}\log\left(\sinh\left(\beta_{N}\left|V_{01}\left(q_{k}\right)\right|\right)\right) & l\neq l'
\end{cases},\label{eq:appendix-rpr-g_k}
\end{equation}
shorthand notation $\overline{l_{i}}=1-l_{i}$ and $\mathcal{Z}_{N}\triangleq\int_{\mathbb{R}^{N}}\int_{\mathbb{R}^{N}}\sum_{\boldsymbol{l}\in\left\{ 0,1\right\} ^{N}}\exp\left[-\beta_{N}H_{N}\left(\boldsymbol{q},\boldsymbol{p},\boldsymbol{l}\right)\right]d\boldsymbol{q}d\boldsymbol{p}$
the normalization constant. 
We leave the detailed derivations and discussions to \cite{Lu}.

\subsection{PIMD-SH: An abstract}

The central idea of PIMD-SH is to construct a reversible Markov process
that the invariant distribution is exactly $\frac{1}{\mathcal{Z}_{N}}\exp\left(-\beta_{N}H_{N}\left(\boldsymbol{q},\boldsymbol{p},\boldsymbol{l}\right)\right)$.
To achieve this, the dynamics of the position and momentum variable
$\boldsymbol{z}\triangleq\left(\boldsymbol{q},\boldsymbol{p}\right)$
is similar to Eqn. \ref{eq:1-level-underdamped-langevin-equation}
as
\begin{equation}
\begin{cases}
d\boldsymbol{q} & =\nabla_{\boldsymbol{p}}H_{N}\left(\boldsymbol{q}\left(t\right),\boldsymbol{p}\left(t\right),\boldsymbol{l}\left(t\right)\right)dt\\
d\boldsymbol{p} & =-\nabla_{\boldsymbol{q}}H_{N}\left(\boldsymbol{q}\left(t\right),\boldsymbol{p}\left(t\right),\boldsymbol{l}\left(t\right)\right)dt-\gamma\boldsymbol{p}dt+\sqrt{\frac{2\gamma M}{\beta_{N}}}d\boldsymbol{B}
\end{cases}.\label{eq:2-level-underdamped-langevin-equation-q,p}
\end{equation}

The evolution of $\boldsymbol{l}\left(t\right)$ follows a surface
hopping type dynamics, essentially a Q-process: for given index configuration
$\boldsymbol{l}$, we restrict that the post-changing state can only
be $\overline{\boldsymbol{l}}$ or those $\boldsymbol{l}'$ which
has distance 1 to $\boldsymbol{l}$. The hopping intensity from $\boldsymbol{l}$
to other possible states is determined by the following detailed balance
condition
\[
p_{\boldsymbol{l}',\boldsymbol{l}}\left(\boldsymbol{z}\right)\exp\left[-\beta_{N}H_{N}\left(\boldsymbol{z},\boldsymbol{l}\right)\right]=p_{\boldsymbol{l},\boldsymbol{l}'}\left(\boldsymbol{z}\right)\exp\left[-\beta_{N}H_{N}\left(\boldsymbol{z},\boldsymbol{l}'\right)\right]
\]
up to an overall scaling parameter $\boldsymbol{\eta}$.

In \cite{Lu} the choice for $p_{\boldsymbol{l}',\boldsymbol{l}}\left(\boldsymbol{z}\right)$
was $\exp\left[\frac{\beta_{N}}{2}\left(H_{N}\left(\boldsymbol{z},\boldsymbol{l}\right)-H_{N}\left(\boldsymbol{z},\boldsymbol{l}'\right)\right)\right]$,
which is not upper-bounded, causing numerical difficulties if the
overall hopping probability exceeds 1. Hence, in this article we propose
the Hasting-style transition intensity $\min\left\{ 1,\exp\left[\frac{\beta_{N}}{2}\left(H_{N}\left(\boldsymbol{z},\boldsymbol{l}\right)-H_{N}\left(\boldsymbol{z},\boldsymbol{l}'\right)\right)\right]\right\} $.

Once the theory for dynamics is fully established, we can apply the
theorem of ergodicity
\[
\langle\widehat{A}\rangle =\lim_{T\to\infty}\frac{1}{T}\int_{0}^{T}W_{N}\left[A\right]\left(\widetilde{z}\left(t\right)\right)dt
\]
and calculate the thermal average by truncating the formula above
with a finite time span.

\section{Bayesian Inversion Supplements}

\subsection{Distance between measures}
\begin{defn}
\label{def:total-variation-distance}The total variation distance
between two measures $\mu$ and $\mu'$ is 
\[
d_{TV}\left(\mu,\mu'\right)\triangleq\frac{1}{2}\int\left|\frac{d\mu}{d\nu}-\frac{d\mu'}{d\nu}\right|d\nu.
\]
\end{defn}

\begin{defn}
\label{def:Hellinger-distance}The Hellinger distance between $\mu$
and $\mu'$ is 
\[
d_{Hell}\left(\mu,\mu'\right)\triangleq\sqrt{\frac{1}{2}\int\left(\sqrt{\frac{d\mu}{d\nu}}-\sqrt{\frac{d\mu'}{d\nu}}\right)^{2}d\nu}.
\]
\end{defn}

\begin{lem}
\label{lem:tv-hellinger-inequality}The total variation metric and
the Hellinger metric are related by the following inequalities
\[
\frac{1}{\sqrt{2}}d_{TV}\left(\mu,\mu'\right)\le d_{Hell}\left(\mu,\mu'\right)\le d_{TV}\left(\mu,\mu'\right)^{\frac{1}{2}}.
\]
\end{lem}

\subsection{Theorem on conditional variables}
\begin{thm}
\label{thm:bayesian-inversion-conditional-variable}Let $\left(\mathcal{X},A\right)$
and $\left(\mathcal{Y},B\right)$ denote a pair of measurable spaces
and let $\nu$ and $\pi$ be probability measures on $\mathcal{X}\times\mathcal{Y}$.
We assume that $\nu\ll\pi$. Thus there exists $\pi$-measurable $\phi:\mathcal{X}\times\mathcal{Y}\to\mathbb{R}$
with $\phi\in L_{\pi}^{1}$ and 
\[
\dfrac{d\nu}{d\pi}\left(x,y\right)=\phi\left(x,y\right).
\]
Assume that the conditional random variable $x|y$ exists under $\pi$
with probability distribution denoted by $\pi^{y}\left(dx\right)$.
Then the conditional random variable $x|y$ exists under $\nu$, with
probability distribution denoted by $\nu^{y}\left(dx\right)$. Furthermore,
$\nu^{y}\ll\pi^{y}$ and if $c\left(y\right)\triangleq\int_{X}\phi\left(x,y\right)d\pi^{y}\left(x\right)>0$,
then
\[
\dfrac{d\nu^{y}}{d\pi^{y}}\left(x\right)=\frac{1}{c\left(y\right)}\phi\left(x,y\right).
\]
\end{thm}

The theorem can be found in \cite{Dashti2017}, Page 28.

\subsection{Well-posedness result}
\begin{notation*}
In this section, let $\mathcal{X}$ and $\mathcal{Y}$ be two separable
Banach spaces, and $\mu_{0}$ a measure on $\mathcal{X}$. Given $y\in\mathcal{Y}$,
$\mu^{y}$ is defined as the measure absolutely continuous to $\mu_{0}$
by
\begin{align*}
\dfrac{d\mu^{y}}{d\mu_{0}}\left(x\right) & =\frac{1}{Z\left(y\right)}\exp\left(-\Phi\left(x;y\right)\right)\\
Z\left(y\right) & =\int_{\mathcal{X}}\exp\left(-\Phi\left(x;y\right)\right)\mu_{0}\left(dx\right).
\end{align*}
\end{notation*}
\begin{assumption}
\label{assu:phi-regularity}Let $\mathcal{X}'$ is a separable subspace
of $\mathcal{X}$ and assume that $\Phi\in C\left(\mathcal{X}'\times\mathcal{Y};\mathbb{R}\right)$.
Assume further that there are functions $M_{i}:\mathbb{R}^{+}\times\mathbb{R}^{+}\to\mathbb{R}^{+},i=1,2$,
monotonic non-decreasing separately in each argument, and with $M_{2}$
strictly positive, such that for all $u\in\mathcal{X}',y,y_{1},y_{2}\in B_{\mathcal{Y}}\left(0,r\right),$
\[
\Phi\left(u;y\right)\geq-M_{1}\left(r,\left|\left|u\right|\right|_{\mathcal{X}}\right),
\]
\[
\left|\Phi\left(u;y_{1}\right)-\Phi\left(u;y_{2}\right)\right|\leq M_{2}\left(r,\left|\left|u\right|\right|_{\mathcal{X}}\right)\left|\left|y_{1}-y_{2}\right|\right|_{\mathcal{Y}}.
\]
\end{assumption}

\begin{thm}
\label{thm:well-posedness}Assume that Assu. \ref{assu:phi-regularity}
holds. Assume that the prior measure is a.s. restricted on the separable
subspace $\mathcal{X}'$ (i.e. $\mu_{0}\left(\mathcal{X}'\right)=1$)
and that $\mu_{0}\left(\mathcal{X}'\cap S\right)>0$ for some bounded
set $S$ in $\mathcal{X}$. Assume additionally that, for every fixed
$r>0$,
\[
\exp\left(M_{1}\left(r,\left|\left|u\right|\right|_{\mathcal{X}}\right)\right)\left(1+M_{2}\left(r,\left|\left|u\right|\right|_{\mathcal{X}}\right)^{2}\right)\in L_{\mu_{0}}^{1}\left(\mathcal{X};\mathbb{R}\right).
\]
Then we have the following asymptotic inequality
\[
d_{Hell}\left(\mu^{y},\mu^{y'}\right)\apprle_{r}\left|\left|y-y'\right|\right|_{\mathcal{Y}}.
\]
\end{thm}

The proof can be found in \cite{Dashti2017}, Page 38-39. Thm. \ref{thm:well-posedness}
ensures that the distance between induced posterior probability measures
is bounded by the distance between different dependent variables.
\begin{lem}
\label{lem:continuity-linfty-observable}Given a function $f\in L_{\mu^{y}}^{\infty}\left(X;\mathbb{R}\right)\cap L_{\mu^{y'}}^{\infty}\left(X;\mathbb{R}\right)$,
we have 
\[
\left|\mathbb{E}^{\mu^{y}}f\left(u\right)-\mathbb{E}^{\mu^{y'}}f\left(u\right)\right|\le C\left(f\right)d_{TV}\left(\mu^{y},\mu^{y'}\right),
\]
which leads to 
\[
\left|\mathbb{E}^{\mu^{y}}f\left(u\right)-\mathbb{E}^{\mu^{y'}}f\left(u\right)\right|\apprle_{r}\left|\left|y-y'\right|\right|_{\mathcal{Y}}
\]
for fixed $f$.
\end{lem}

\begin{proof}
By the definition of $d_{TV}$, there exists a measure $\nu$ s.t.
both $\mu^{y}$ and $\mu^{y'}$ are absolutely continuous w.r.t. $\nu$;
for example, take $\nu=\frac{1}{2}\left(\mu^{y}+\mu^{y'}\right)$.
Then a direct calculation shows that
\begin{align*}
\left|\mathbb{E}^{\mu^{y}}f\left(u\right)-\mathbb{E}^{\mu^{y'}}f\left(u\right)\right| & =\left|\int_{X}f\ d\mu^{y}\left(u\right)-\int_{X}f\ d\mu^{y'}\left(u\right)\right|\\
 & =\left|\int_{X}f\dfrac{d\mu^{y}}{d\nu}\thinspace d\nu\left(u\right)-\int_{X}f\dfrac{d\mu^{y'}}{d\nu}\thinspace d\nu\left(u\right)\right|\\
 & =\left|\int_{X}f\left(\dfrac{d\mu^{y}}{d\nu}-\dfrac{d\mu^{y'}}{d\nu}\right)\thinspace d\nu\left(u\right)\right|\\
 & \le\left|\left|f\right|\right|_{L_{\nu}^{\infty}}\cdot\int_{X}\left|\dfrac{d\mu^{y}}{d\nu}-\dfrac{d\mu^{y'}}{d\nu}\right|d\nu\left(u\right)\\
 & =\max\left(\left|\left|f\right|\right|_{L_{\mu^{y}}^{\infty}},\left|\left|f\right|\right|_{L_{\mu^{y'}}^{\infty}}\right)\cdot d_{TV}\left(\mu^{y},\mu^{y'}\right).
\end{align*}

The second conclusion is a direct corollary from Thm. \ref{thm:well-posedness}
($d_{Hell}\left(\mu^{y},\mu^{y'}\right)\apprle_{r}\left|\left|y-y'\right|\right|_{\mathcal{Y}}$)
and Lem. \ref{lem:tv-hellinger-inequality} ($d_{TV}\left(\mu^{y},\mu^{y'}\right)\le\sqrt{2}d_{Hell}\left(\mu^{y},\mu^{y'}\right)$).
\end{proof}

\subsection{Measure preserving dynamics\label{subsec:Measure-preserving-dynamics}}

The following algorithm from \cite{Dashti2017} is proposed to sample
distributions in the form of $\mu\left(du\right)=\mu_{0}\left(du\right)\frac{1}{Z}\exp\left(-\Phi\right)$:

\begin{algorithm}
\begin{algor}[1]
\item [{{*}}] Given the choice function $a:X\times X\to\left[0,1\right]$
and proposal kernel $Q\left(u,dv\right)$.
\item [{{*}}] Initialize $k=0$ and first term $u^{\left(0\right)}\in X$
in the sampling sequence.
\item [{while}] stopping criterion is not satisfied
\begin{algor}[1]
\item [{{*}}] Propose $v^{\left(k\right)}$ based on $Q\left(u^{\left(k\right)},dv\right)$.
\item [{{*}}] Pick a random number $r^{\left(k\right)}\sim\mathcal{U}\left[0,1\right]$,
independently of $\left(u^{\left(k\right)},v^{\left(k\right)}\right)$.
\item [{if}] $r^{\left(k\right)}<a\left(u^{\left(k\right)},v^{\left(k\right)}\right)$
\begin{algor}[1]
\item [{{*}}] Accept the proposal: $u^{\left(k+1\right)}=v^{\left(k\right)}$.
\end{algor}
\item [{else}]~
\begin{algor}[1]
\item [{{*}}] Reject the proposal: $u^{\left(k+1\right)}=u^{\left(k\right)}$.
\end{algor}
\item [{endif}]~
\end{algor}
\item [{endwhile}]~
\end{algor}
\caption{Metropolis-Hasting Algorithm Adapted to Infinite-dimensional Systems}

\label{alg:Metropolis-Hasting-Algorithm-Infinite-Dimension}
\end{algorithm}

To preserve the desired measure, we can choose the choice function
$a$ and proposal kernel $Q$ according to the following proposition.
\begin{prop}
\label{assu:mcmc-reversible-coondition}The Markov chain sampled in
Alg. \ref{alg:Metropolis-Hasting-Algorithm-Infinite-Dimension} is
reversible to the measure $\mu\sim\mu_{0}\exp\left(-\Phi\right)$
if the following assumption holds: the negative log potential $\Phi:X\to\mathbb{R}$
is upper-bounded on any bounded set of $X$, the choice function $a$$\left(u,v\right)=\min\left\{ 1,\exp\left(\Phi\left(u\right)-\Phi\left(v\right)\right)\right\} $
and the proposal kernel satisfies the reversible condition
\[
\mu_{0}\left(du\right)Q\left(u,dv\right)=\mu_{0}\left(dv\right)Q\left(v,du\right).
\]
\end{prop}

\begin{note*}
The proof is on Page 53 in \cite{Dashti2017}.
\end{note*}

\end{document}